\documentclass{amsart}

\numberwithin{equation}{section}

\usepackage{hyperref}
\usepackage{amsmath}
\usepackage{amsthm}
\usepackage{amssymb}
\usepackage{amsxtra}
\usepackage{enumerate}
\usepackage{amscd}
\usepackage[alphabetic]{amsrefs}
\usepackage[capitalize]{cleveref}

\theoremstyle{plain}
\newtheorem{thm}{Theorem}[section]
\newtheorem{thm-dfn}[thm]{Theorem-Definition}
\newtheorem{prop}[thm]{Proposition}
\newtheorem{prop-dfn}[thm]{Proposition-Definition}
\newtheorem{lem}[thm]{Lemma}
\newtheorem{cor}[thm]{Corollary}

\theoremstyle{definition}

\newtheorem{rem}[thm]{Remark}

\newtheorem{dfn}[thm]{Definition}

\newtheorem{notn}[thm]{Notation}

\newcommand\fg{\ensuremath{\mathfrak g}}

\newcommand\cC{\ensuremath{\mathcal C}}

\newcommand\bbC{\ensuremath{\mathbb C}}

\newcommand\bbR{\ensuremath{\mathbb R}}

\usepackage{amscd}
\usepackage{tikz}
\usetikzlibrary{matrix,arrows}

\newcommand{\cop}{\Delta}
\newcommand{\counit}{\varepsilon}
\renewcommand{\phi}{\varphi}
\newcommand{\cvect}{\mathbf{\bbC Vect}}
\newcommand{\cvectf}{\mathbf{\bbC Vect}_f}
\newcommand{\hmod}{\mathbf{HMod}}
\newcommand{\hmodf}{\mathbf{HMod}_f}
\newcommand{\hmods}{\hmod^*}
\newcommand{\hmodsp}{\hmod_+^*}
\newcommand{\subcat}{\cC}
\newcommand{\us}[1]{{}^* #1}
\newcommand{\ue}[1]{{}^{e} #1}
\DeclareMathOperator{\ev}{ev}

\DeclareMathOperator{\End}{End}
\DeclareMathOperator{\Hom}{Hom}
\DeclareMathOperator{\cHom}{cHom}

\DeclareMathOperator{\id}{id}
\DeclareMathOperator{\im}{im}
\DeclareMathOperator{\tr}{tr}
\newcommand{\homH}{\Hom_H}
\newcommand{\homl}{\Hom_\ell}
\newcommand{\homr}{\Hom_r}
\newcommand{\cHomH}{\cHom_H}

\newcommand{\Endl}{\End_\ell}
\newcommand{\Endr}{\End_r}
\renewcommand{\rhd}{\triangleright}
\newcommand{\brhd}{\blacktriangleright}
\renewcommand{\bar}{\overline}
\newcommand{\inprod}[1][\, , \,]{\left\langle #1 \right\rangle}
\newcommand{\inprodv}[1][\, , \,]{\left\langle #1 \right\rangle_V}
\newcommand{\inprodw}[1][\, , \,]{\left\langle #1 \right\rangle_W}
\newcommand{\rinprod}[1][\, , \,]{\left( #1 \right)}
\newcommand{\rinprodv}[1][\, , \,]{\left( #1 \right)_V}
\newcommand{\rinprodw}[1][\, , \,]{\left( #1 \right)_W}
\renewcommand{\epsilon}{\varepsilon}
\newcommand{\vth}{\vartheta}
\newcommand{\set}[1]{\{ #1 \}}

\newcommand{\rmat}{\mathcal{R}}

\title{\(\ast\)-structures on module-algebras}
\author{Matthew Tucker-Simmons}
\address{Department of Mathematics \\
University of California \\
Berkeley, CA 94720-3840}
\email{mbtucker@math.berkeley.edu}
\thanks{The research reported here was supported in part by National Science Foundation grant DMS-1066368.}

\begin{document}

\begin{abstract}
  This chapter lays out a framework for discussing \(\ast\)-structures on module-algebras over a Hopf \(\ast\)-algebra \(H\).
  We define a complex conjugation functor \(V \mapsto \bar{V}\), which is an involution on the module category \(\hmod\), and discuss its interaction with natural constructions such as direct sums, duality, Hom, and tensor products.
  We define \(\ast\)-structures first at the level of modules.
  We say that \(V\) is a \(\ast\)-module if there is an isomorphism \(\ast : \bar{V} \to V\) in \(\hmod\) which is involutive in an appropriate sense.
  Then we define \(\ast\)-structures on algebras in \(\hmod\) by requiring compatibility with multiplication.
  We show that a \(\ast\)-structure on a module lifts uniquely to the tensor algebra, and we prove that the tensor algebra has a universal mapping properly for morphisms of \(\ast\)-modules.
  We also discuss inner products and adjoints in this framework.
  Finally, we discuss the interaction between \(\ast\)-structures, \(R\)-matrices, and braidings.
\end{abstract}

\maketitle

\section{Background and notation}
\label{sec:background}

\subsection{Background}
\label{sec:motivation}

My motivation for writing this document was to understand what is the appropriate notion of a $*$-algebra in the category of modules over a Hopf $*$-algebra.
Although this is not very deep, it is perhaps a little tricky to phrase things properly.
Since $*$-structures are antilinear, one is forced to work with antilinear maps of complex vector spaces. 
This is aesthetically displeasing since one is then required to move outside the category of vector spaces and linear maps to deal with what is almost a linear phenomenon.
The framework of complex conjugate linear algebra allows one to phrase everything in terms of linear maps rather than antilinear ones.
It is then relatively straightforward to generalize the notions to modules (and module-algebras) over a Hopf $*$-algebra.

The structure of this document is as follows.
In the remainder of Section \ref{sec:background} we set notation.
In Section \ref{sec:hopfstar-background} we give the necessary background on Hopf $*$-algebras and discuss some properties of their module categories.
In Section \ref{sec:linalg} we discuss antilinear maps, define complex conjugation as an endofunctor on the category of complex vector spaces, and show that this functor is natural with respect to many common operations on linear spaces.
In Section \ref{sec:conj-of-modules} we extend the notions from Section \ref{sec:linalg} to modules over a Hopf $*$-algebra, and in Section \ref{sec:conj-of-algebras} we examine what these notions mean for module-algebras.
In Section \ref{sec:star-modules} we define \(\ast\)-structures on modules, and in Section \ref{sec:star-algebras} we extend this to module-algebras.
In Section \ref{sec:inner-products} we treat inner products and adjoints of linear maps in our framework.
Finally, in Section \ref{sec:stars-and-braidings} we discuss the interaction between \(\ast\)-structures, \(R\)-matrices, and braidings.

\subsection{Notation}
\label{sec:notation}

All vector spaces are over $\bbC$.
We denote the category of vector spaces over $\bbC$ with linear maps by $\cvect$, and we denote the full subcategory of finite-dimensional vector spaces by $\cvectf$.
The linear dual of a vector space $V$ will be denoted $V^*$.  
In order to avoid a proliferation of $*$'s, of which there are quite enough already, we denote the transpose (or dual map) of a linear map $T : V \to W$ by $T^{\tr} : W^* \to V^*$.
The vector space of all linear maps from $V$ to $W$ will be denoted $\Hom(V,W)$.
An undecorated $\Hom$ will always refer to linear maps, not module maps.

For the rest of this document, $H$ will denote a Hopf algebra over $\bbC$ with coproduct $\cop$, counit $\counit$, and antipode $S$.
We use the Sweedler notation
\[  \cop(a) = a_{(1)} \otimes a_{(2)} \]
for the coproduct of $H$, with implied summation.
When we refer to $H$-modules, we mean left modules for the underlying algebra structure of $H$.
We denote the category of left $H$-modules with $H$-module maps by $\hmod$, and we denote the full subcategory of finite-dimensional modules by $\hmodf$.
The vector space of all module morphisms from $V$ to $W$ will be denoted by $\homH(V,W)$.
The vector space $\Hom(V,W)$ of all linear maps from $V$ to $W$ has the structure of an $H$-module, but as we will see in Section \ref{sec:hopfstar-background} there are two choices for the action when $H$ is a Hopf $*$-algebra and one must distinguish between them.

For any category $\cC$, we write $X \in \cC$ to mean that $X$ is an object of $\cC$.

\section{Hopf $*$-algebras and their modules}
\label{sec:hopfstar-background}

In this section we recall some necessary definitions and facts about Hopf $*$-algebras.
One important (and not immediately obvious) fact is that the antipode of a Hopf $*$-algebra is invertible.
This has the consequence that some linear spaces--namely the linear dual $\Hom(V,\bbC)$ of a module $V$ and the space $\Hom(V,W)$ of linear maps between two modules--carry two natural actions of $H$.
We assume throughout the rest of this document that $H$ is a Hopf $*$-algebra (see Definition \ref{dfn:hopfstar-alg}), although we note that in \S \ref{sec:hopfstar-duals} and \S \ref{sec:hopfstar-hom} we only require the antipode to be invertible.

Hopf $*$-algebras are discussed, for example, in Section 1.2.7 of \cite{KliSch97} and in Section 4.1.F of \cite{ChaPre95}.
Tensor products and duals of modules, as well as the module structure on $\Hom$-spaces, can be found in Section 4.1.C of \cite{ChaPre95}.

\begin{notn}
  The symbol $\rhd$ will generally indicate the action of a Hopf algebra on a module, so that $a \rhd v$ means the action of $a$ on $v$.
  When we consider more than one action of $H$ on the same vector space, we will distinguish one of the actions by using the symbol $\brhd$ rather than $\rhd$.
\end{notn}

\subsection{Hopf $*$-algebras}
\label{sec:hopfstar-algebras}

\begin{dfn}
  \label{dfn:hopfstar-alg}
  A \emph{$*$-structure} on $H$ is an antilinear map $* : H \to H$ such that
  \begin{enumerate}[(a)]
  \item $*$ is involutive: $(a^*)^* = a$ for all $a \in H$.
  \item $*$ reverses products: $(ab)^* = b^*a^*$ for all $a,b \in H$.
  \item $\cop$ is a $*$-homomorphism: $\cop(a^*) = a_{(1)}^* \otimes a_{(2)}^*$ for all $a \in H$.
  \end{enumerate}
  If $H$ is equipped with a $*$-structure then we say that $H$ is a \emph{Hopf $*$-algebra}.
\end{dfn}

It follows from Definition \ref{dfn:hopfstar-alg} that the counit is also a $*$-homomorphism, i.e.\ that $\counit(a^*) = \bar{\counit(a)}$ for $a \in H$.
Another consequence of the definition is that
\begin{equation}
  \label{eq:sstar}
  * \circ S \circ * \circ S  = \id_H,
\end{equation}
which is shown in Proposition 10, Section 1.2.7 of \cite{KliSch97}.
The proof proceeds by showing that $* \circ S \circ *$ is an antipode for $H^{op}$ (which \emph{a priori} is only a bialgebra), and hence $S$ is invertible with $S^{-1} = * \circ S \circ *$ by Proposition 6, Section 1.2.4 of the same reference.

\begin{rem}
  \label{rem:sstar}
  The fact that the antipode is necessarily invertible shows that not every Hopf algebra over $\bbC$ can be given a $*$-structure.
  In \cite{Tak71}, a so-called free Hopf algebra $H(C)$ is constructed for any coalgebra $C$, and those coalgebras for which the antipode of $H(C)$ is invertible are classified.  
  In particular, if $C = M_n(\bbC)^*$ is the linear dual of a matrix algebra with $n > 1$, then the antipode of $H(C)$ is not invertible, and hence $H(C)$ cannot be equipped with a $*$-structure.
  
  However, there are also examples of Hopf algebras over $\bbC$ with bijective antipode which cannot be endowed with a $*$-structure.
  It is shown in Example 10, Section 1.2.7 of \cite{KliSch97} that if $\fg$ is a complex Lie algebra, then $*$-structures on the universal enveloping algebra $U(\fg)$ correspond bijectively to real forms of $\fg$.
  It is shown in Example 1, Section 1.7.1 of \cite{Vin94} that the complex Lie algebra spanned by $X,Y$, and $Z$ with relations
  \[  [X,Y] = Y, \quad [X,Z] = 2iZ, \quad [Y,Z] = 0   \]
  has no real form, and hence its enveloping algebra has no $*$-structures.
\end{rem}

\subsection{Tensor products of modules}
\label{sec:hopfstar-tensor}

We recall that the coproduct of $H$ allows us to form the tensor product of modules.
For $V,W \in \hmod$ we endow $V \otimes W$ with the tensor product action
\begin{equation}
  \label{eq:tensorprod_action}
  a \rhd (v \otimes w) = (a_{(1)} \rhd v ) \otimes (a_{(2)} \rhd w).
\end{equation}
This gives $\hmod$ the structure of a monoidal category in which the monoidal unit is the ground field $\bbC$, which is an $H$-module via the counit $\counit$.

\subsection{Duals of modules}
\label{sec:hopfstar-duals}

For any vector space $V$ we can form the dual vector space $\Hom(V,\bbC)$.
If $V$ is a left module for a complex algebra $A$, then $\Hom(V,\bbC)$ is naturally a right $A$-module, or equivalently a left $A^{op}$-module.
For a Hopf algebra $H$, the antipode can be viewed as an algebra homomorphism $S : H \to H^{op}$.
Hence for $V \in \hmod$, the linear space $\Hom(V,\bbC)$ becomes a left $H$-module via the action
\begin{equation}
  \label{eq:dual-action}
  (a \rhd f)(v) = f(S(a) \rhd v)
\end{equation}
for all $a \in H$, $f \in \Hom(V,\bbC)$, and $v \in V$.

We have seen in \S \ref{sec:hopfstar-algebras} that the antipode $S$ is invertible since $H$ is a Hopf $*$-algebra.
Thus $S^{-1}$ gives a second algebra homomorphism $H \to H^{op}$, and corresponding to this is a second left $H$-module structure on $\Hom(V,\bbC)$ given by
\begin{equation}
  \label{eq:black-dual-action}
  (a \brhd f)(v) = f(S^{-1}(a) \rhd v)
\end{equation}
for all $a \in H$, $f \in \Hom(V,\bbC)$, and $v \in V$.
As this action is less well-known than the standard dual action \eqref{eq:dual-action}, we briefly verify that it is an action.
Linearity in $a$ and $f$ is clear, so we just need to check compatibility with multiplication in $H$.
For $a,b \in H$ and $f \in \Hom(V, \bbC)$ we have
\begin{align*}
  (a \brhd (b \brhd f)) (v) & = (b \brhd f)(S^{-1}(a) \rhd v) \\
  & = f (S^{-1}(b) \rhd (S^{-1}(a) \rhd v)) \\
  & = f ((S^{-1}(b) S^{-1}(a)) \rhd v) \\
  & = f (S^{-1}(ab) \rhd v) \\
  & = (ab \brhd f) (v),
\end{align*}
so $a \brhd (b \brhd f) = ab \brhd f$.

\begin{dfn}
  \label{dfn:dual-actions}
  For $V \in \hmod$, we define $V^* = \Hom(V, \bbC)$ with the action $\rhd$ of $H$ given by \eqref{eq:dual-action}, and we refer to this as the \emph{left dual of $V$}.
  We define $\us{V} = \Hom(V,\bbC)$ with the action $\brhd$ of $H$ given by \eqref{eq:black-dual-action}, and we refer to this as the \emph{right dual of $V$.}
\end{dfn}

\begin{rem}
  \label{rem:eval-maps}
  Lemma \ref{lem:eval-maps} below explains how to remember which dual is the left and which one is the right.
  The left dual, $V^*$, goes to the left of $V$ in the evaluation pairing, while the right dual, $\us{V}$, goes to the right.
  In both cases, the superscript $*$ is adjacent to the $\otimes$ symbol; one can think of the $*$ as ``eating'' the vector.
\end{rem}

\begin{lem}
  \label{lem:eval-maps}
  For $V \in \hmod$, we have
  \begin{enumerate}[(a)]
  \item The evaluation map $\ev_V : V^* \otimes V \to \bbC$ given by $\phi \otimes v \mapsto \phi(v)$ is a morphism of $H$-modules.
  \item The evaluation map ${}_V\ev : V \otimes \us{V} \to \bbC$ given by $v \otimes \phi \mapsto \phi(v)$ is a morphism of $H$-modules.
  \end{enumerate}
\end{lem}

\begin{proof}
  \begin{enumerate}[(a)]
  \item For $f \in V^*$, $v \in V$, $a \in H$ we have
    \begin{align*}
      \ev_V(a \rhd (f \otimes v)) & = (a_{(1)} \rhd f)(a_{(2)} \rhd v) \\
      & = f(S(a_{(1)})a_{(2)} \rhd v) \\
      & = \counit(a)f(v) \\
      & = a \rhd (\ev_V(f \otimes v)).   
    \end{align*}
  \item Similar to $\mathrm{(a)}$. \qedhere
  \end{enumerate}
\end{proof}

\begin{prop}
  \label{prop:double-dual}
  Let $V \in \hmod$, and for $v \in V$ let $\delta_v : \Hom(V,\bbC) \to \bbC$ be the evaluation map $\delta_v(f) = f(v)$.
  \begin{enumerate}[(a)]
  \item The linear map $V \to \us{(V^*)}$ given by $v \mapsto \delta_v$ is a morphism of $H$-modules, and it is an isomorphism for $V \in \hmodf$.
  \item The linear map $V \to (\us{V})^*$ given by $v \mapsto \delta_v$ is a morphism of $H$-modules, and it is an isomorphism for $V \in \hmodf$.
  \end{enumerate}
\end{prop}

\begin{proof}
  \begin{enumerate}[(a)]
  \item On the one hand, for $a \in H$, $v \in V$ and $f \in V^*$ we have
    \[  \delta_{a \rhd v}(f) = f(a \rhd v),   \]
    while on the other hand we have
    \begin{align*}
      (a \brhd \delta_v) (f) & = \delta_v (S^{-1}(a) \rhd f) \\
      & = (S^{-1}(a) \rhd f)(v) \\
      & = f (S(S^{-1}(a)) \rhd v  ) \\
      & = f(a \rhd v),
    \end{align*}
    and hence $\delta_{a \rhd v} = a \brhd \delta_v$.
    If $V$ is finite-dimensional then the dimensions of $V$ and $\us{(V^*)}$ coincide.  
    Since $v \mapsto \delta_v$ is clearly injective, it is also surjective, and hence is an isomorphism.
  \item Similar to $\mathrm{(a)}$. \qedhere
  \end{enumerate}
\end{proof}

In some Hopf algebras, e.g.\ quasitriangular ones, the square of the antipode is an inner automorphism; see Proposition 5, Section 8.1.3 of \cite{KliSch97}.
In that case the left and right duals coincide, and then for finite-dimensional modules the the second dual is naturally isomorphic to the original module:

\begin{prop}
  \label{prop:ssquared-inner}
  Suppose that there is an invertible element $u$ of $H$ such that $S^2(a) = uau^{-1}$ for all $a \in A$.  Then:
  \begin{enumerate}[(a)]
  \item $au \rhd f = ua \brhd f$ for all $f \in \Hom(V,\bbC)$ and $a \in H$.  In particular $u \rhd f = u \brhd f$.
  \item The linear map $\us{V} \to V^*$ given by $f \mapsto u \rhd f$ is an isomorphism of modules whose inverse is $f \mapsto u^{-1} \rhd f$.
  \item The linear map $V \to V^{**}$ given by $v \mapsto \delta_{u \rhd v}$ is a morphism of modules, and it is an isomorphism if $V \in \hmodf$, where $\delta$ is as defined in Proposition \ref{prop:double-dual}.
  \end{enumerate}
\end{prop}

\begin{proof}
  \begin{enumerate}[(a)]
  \item  First note that for $a \in H$ we have 
    \[ S(au) = S^{-1}(S^2(au)) = S^{-1}(ua),  \]
    so for any $v \in V$ we have
    \begin{align*}
      (au \rhd f)(v) & = f(S(au) \rhd v) \\
      & = f (S^{-1}(ua) \rhd v) \\
      & = (ua \brhd f)(v),
    \end{align*}
    so $au \rhd f = ua \brhd f$.
  \item  We compute
    \[ u \rhd (a \brhd f) = u \brhd (a \brhd f) = ua \brhd f = au \rhd f = a \rhd (u \rhd f),   \]
    using $\mathrm{(a)}$ for the first and third equalities.
    Hence $f \mapsto u \rhd f$ is a module map $\us{V} \to V^*$.
  \item  For $a \in A$ and $v \in V$ we have
    \begin{align*}
      (a \rhd \delta_{u \rhd v})(f) & = \delta_{u \rhd v}(S(a) \rhd f) \\
      & = (S(a) \rhd f) (u \rhd v) \\
      & = f (S^2(a)u \rhd v) \\
      & = f(ua \rhd v) \\
      & = \delta_{u \rhd (a \rhd v)}(f),
    \end{align*}
    so we see that $a \rhd \delta_{u \rhd v} = \delta_{u \rhd (a \rhd v)}$, i.e.\ $v \mapsto \delta_{u \rhd v}$ is a module map.
    The map is injective, and hence bijective if $V$ is finite-dimensional since the dimensions of $V$ and $V^{**}$ coincide.  \qedhere
  \end{enumerate}
\end{proof}

\begin{rem}
  \label{rem:rigidity-etc}
  The existence of left and right duals makes $\hmodf$ into a \emph{rigid} or \emph{autonomous} category; see, for example, Section 2.1 of \cite{BakKir01} for a precise definition.
  If in addition $S^2$ is inner then part $\mathrm{(c)}$ of Proposition \ref{prop:ssquared-inner} implies that $\hmodf$ is a \emph{pivotal} category in the sense of Definition 5.1 of \cite{FreYet92}.
\end{rem}

\subsection{Module structures on $\Hom(V,W)$}
\label{sec:hopfstar-hom}

The two actions of $H$ on $\Hom(V,\bbC)$ give rise to two actions of $H$ on the space $\Hom(V,W)$ for $V,W \in \hmod$.  The standard action is given by
\begin{equation}
  \label{eq:hom-action}
  a \rhd T = a_{(1)}T S(a_{(2)})
\end{equation}
for $T \in \Hom(V,W)$ and $a \in H$.  
It is not difficult to check directly that \eqref{eq:hom-action} defines an action of $H$.
However, the definition can be better motivated by viewing the space of linear maps as
\begin{equation}
  \label{eq:left-hom-decomp}
  \Hom(V,W) \simeq W \otimes V^*,
\end{equation}
where a simple tensor $w \otimes \phi \in W \otimes V^*$ acts as a linear map on a vector $v \in V$ by
\[  (w \otimes \phi)(v) = w \phi(v).    \]
Note that we have written the scalar on the right here.
Then it is straightforward to see that
\begin{align*}
  [a \rhd (w \otimes \phi)](v) & = [(a_{(1)} \rhd w) \otimes (a_{(2)} \rhd \phi)](v) \\  
  & = (a_{(1)} \rhd w) (a_{(2)} \rhd \phi)(v) \\
  & = a_{(1)} \rhd [w \phi (S(a_{(2)}) \rhd v )] \\
  & = [a_{(1)} (w \otimes \phi) S(a_{(2)})] (v),
\end{align*}
which explains \eqref{eq:hom-action}.

\begin{rem}
  \label{rem:hom-decomp-caveat}
  We should note that the decomposition \eqref{eq:left-hom-decomp} holds only when at least one of $V$ and $W$ is finite-dimensional.
  In general there is an injective linear map $W \otimes V^* \to \Hom(V,W)$ whose range consists of all finite-rank operators.
  The point here is that we only use the correspondence \eqref{eq:left-hom-decomp} heuristically; the action \eqref{eq:hom-action} is well-defined no matter whether $V$ or $W$ is finite-dimensional.
  The same remarks apply to the decomposition \eqref{eq:right-hom-decomp} and the action \eqref{eq:black-hom-action} below as well.
\end{rem}

In order to get the tensor product action to agree with \eqref{eq:hom-action} it was essential that we used $V^*$ rather than $\us{V}$ in the decomposition \eqref{eq:left-hom-decomp}.
Also note that the order of the two factors in this decomposition was arranged so that evaluation of linear maps on vectors is given by
\begin{equation}
  \label{eq:hom-eval}
  \id_W \otimes \ev_V : W \otimes V^* \otimes V \to W,
\end{equation}
which itself is a morphism of modules.
We prove this in general (i.e.~for modules which are not necessarily finite-dimensional) in Proposition \ref{lem:hom-eval-maps}.

This tells us how we can get another canonical action of $H$ on $\Hom(V,W)$; we exchange $\us{V}$ for $V^*$ and swap the order of the factors.  Thus we write
\begin{equation}
  \label{eq:right-hom-decomp}
  \Hom(V,W) \simeq \us{V} \otimes W,
\end{equation}
with the action of a simple tensor $\phi \otimes w \in \us{V} \otimes W$ on $v \in V$ defined by
\[  (\phi \otimes w)(v) = \phi(v)w.    \]
\begin{rem}
  If we were bolder we would write $v$ on the left of $\phi \otimes w$ in the preceding display, but as the overwhelming convention is to write functions to the left of their arguments, we seem to be stuck with the current clunky formulation.
\end{rem}
In any case, the decomposition \eqref{eq:right-hom-decomp} allows us to define another action of $H$ on $\Hom(V,W)$, which we denote using the symbol $\brhd$.  
For $a \in H$, $\phi \otimes w \in \us{V} \otimes W$ and $v \in V$, the tensor product action gives
\begin{align*}
  [a \brhd (\phi \otimes w)](v) & = [(a_{(1)} \brhd \phi) \otimes (a_{(2)} \rhd w)](v) \\
  & = (a_{(1)} \brhd \phi)(v) a_{(2)}\rhd w \\
  & = a_{(2)} \rhd [\phi(S^{-1}(a_{(1)}) \rhd v ) w] \\
  & = [a_{(2)}(\phi \otimes w) S^{-1}(a_{(1)})](v).
\end{align*}
Hence for $T \in \Hom(V,W)$ the action $\brhd$ of $H$ is given by
\begin{equation}
  \label{eq:black-hom-action}
  a \brhd T = a_{(2)} T S^{-1}(a_{(1)}).
\end{equation}

\begin{dfn}
  \label{dfn:hom-actions}
  For $V,W \in \hmod$, we define $\homl(V,W)$ to be the linear space $\Hom(V,W)$ with the action $\rhd$ given by~\eqref{eq:hom-action}, and we call this the \emph{left $\Hom$-space}.  
  We define $\homr(V,W)$ to be the linear space $\Hom(V,W)$ with the action $\brhd$ given by \eqref{eq:black-hom-action}, and we call this the \emph{right $\Hom$-space}.
  When $V = W$, we denote $\homl(V,V)$ and $\homr(V,V)$ by $\Endl(V)$ and $\Endr(V)$, respectively.
\end{dfn}

We now formalize the statements made in Remark \ref{rem:hom-decomp-caveat}:

\begin{prop}
  \label{prop:hom-decomps}
  Let $V,W \in \hmod$.
  \begin{enumerate}[(a)]
  \item The map $W \otimes V^* \to \homl(V,W)$ given by $(w \otimes \phi)(v) = w \phi(v)$ for $v \in V$, $w \in W$, and $\phi \in V^*$ is an injective module map.
    It is an isomorphism if $V$ and $W$ are finite-dimensional.
  \item The map $\us V \otimes W \to \homr(V,W)$ given by $(\phi \otimes w)(v) = \phi(v)w$ for $v \in V$, $w \in W$, and $\phi \in V^*$ is an injective module map.
    It is an isomorphism if $V$ and $W$ are finite-dimensional.
  \end{enumerate}
\end{prop}

\begin{proof}
  The proofs that these are module maps are exactly the displayed computations immediately preceding Remark \ref{rem:hom-decomp-caveat} and Definition \ref{dfn:hom-actions}, respectively.
  Injectivity is straightforward, and then these maps are isomorphisms for finite-dimensional modules just by a dimension count.
\end{proof}

\begin{rem}
  \label{rem:hom-side-clarification}
  As with duals, there is a simple way to to remember which is the left $\Hom$-space and which is the right.
  The left $\Hom$-space $\homl(V,W)$ goes to the left of $V$ in the evaluation map, and it is constructed using the left dual of $V$.
  The right $\Hom$-space $\homr(V,W)$ goes to the right of $V$ in the evaluation map, and it is constructed using the right dual of $V$.
  This is encapsulated in:
\end{rem}

\begin{lem}
  \label{lem:hom-eval-maps}
  For $V,W \in \hmod$, we have
  \begin{enumerate}[(a)]
  \item The evaluation map $\homl (V,W) \otimes V \to W$ given by $T \otimes v \mapsto T(v)$ is a morphism of $H$-modules.
  \item The evaluation map $V \otimes \homr (V,W) \to W$ given by $v \otimes T \mapsto T(v)$ is a morphism of $H$-modules.
  \end{enumerate}
\end{lem}

\begin{proof}
  We do part (a) only; part (b) is similar.
  For finite-dimensional modules the result follows from Equation \eqref{eq:hom-eval}.
  We now give a direct proof which works for arbitrary modules.
  Denoting the evaluation map by $\ev$, for $a \in H$ we have
  \begin{align*}
    \ev \left( a \rhd (T \otimes v)    \right) & = \ev \left( (a_{(1)} \rhd T) \otimes (a_{(2)} \rhd v)    \right)   \\
    & = a_{(1)} \rhd T(S(a_{(2)})a_{(3)} \rhd v) \\
    & = a_{(1)} \rhd T(\counit(a_{(2)})v) \\
    & = a \rhd \left( Tv  \right) \\
    & = a \rhd \ev(T \otimes v).  \qedhere
  \end{align*}
\end{proof}

An element $v$ of an $H$-module $V$ is called \emph{invariant} if $a \rhd v = \counit(a)v$ for all $a \in H$.  The submodule of invariant elements is denoted by $V^H$.
For the standard action \eqref{eq:hom-action} of $H$ on $\Hom(V,W)$, it is well-known that the invariant elements are precisely the $H$-module maps.
It is therefore natural to ask what the invariants are for the action \eqref{eq:black-hom-action}.  
It turns out that they are the same:

\begin{prop}
  \label{prop:hom-invariants}
  For $V,W \in \hmod$ we have
  \[ \homl(V,W)^H = \homH(V,W) = \homr(V,W)^H.   \]
\end{prop}

\begin{proof}
  First suppose that $T$ is a module map.  Then for $a \in H$ we have
  \[ a \rhd T = a_{(1)} T S(a_{(2)}) = a_{(1)}S(a_{(2)}) T = \counit(a)T,   \]
  and similarly we see that
  \[ a \brhd T = a_{(2)} T S^{-1}(a_{(1)}) = a_{(2)}S^{-1}(a_{(1)})T = \counit(a) T;   \]
  the last equality follows since $S(a_{(2)}S^{-1}(a_{(1)})) = a_{(1)}S(a_{(2)}) = \counit(a)$.

  Now suppose that $T \in \homl(V,W)^H$.  Using part $\mathrm{(a)}$ of Lemma \ref{lem:hom-eval-maps}, for $v \in V$ we have
  \begin{align*}
    a \rhd (T(v)) & = (a_{(1)} \rhd T) (a_{(2)} \rhd v) \\
    & = \counit(a_{(1)})T(a_{(2)} \rhd v) \\
    & = T(a \rhd v).
  \end{align*}
  Similarly, if $T \in \homr(V,W)^H$, we have using part $\mathrm{(b)}$ of Lemma \ref{lem:hom-eval-maps} that
  \begin{align*}
    a \rhd (T(v)) & = (a_{(2)}\brhd T)(a_{(1)} \rhd v) \\
    & = \counit(a_{(2)}) T(a_{(1)} \rhd v) \\
    & = T(a \rhd v). \qedhere
  \end{align*}
\end{proof}

\section{Antilinear maps and the conjugation functor on $\cvect$}
\label{sec:linalg}

While our main goal is to describe complex conjugates of modules for Hopf $*$-algebras, these notions also make sense for vector spaces over $\bbC$.  
Actually vector spaces are just the special case of modules over the Hopf $*$-algebra $\bbC$ itself, but anyway$\dots$  

Here we set out the basics of complex conjugate linear algebra before moving on to study modules in Section \ref{sec:conj-of-modules}.
We begin with some elementary remarks on antilinear maps.
Then we introduce the complex conjugation functor, our main technical tool in what follows.
We show that the conjugation functor allows antilinear maps to be interpreted naturally as linear ones.
Finally we discuss several natural constructions in $\cvect$ and examine their interactions with conjugation.
We emphasize the functorial nature of the constructions throughout.

\subsection{Elementary remarks on antilinear maps}
\label{sec:antilinear}

When doing linear algebra over $\bbC$, one often encounters the notion of an antilinear map between vector spaces.
A trivial example is the map $\lambda \mapsto \bar{\lambda}$ of $\bbC$ to itself.
A less trivial example is the map $T \mapsto T^*$, where $T$ is a linear map between complex inner product spaces.
Another example is a complex inner product itself, which is antilinear in the first variable.

While these maps are generally not difficult to deal with on their own, the framework of complex conjugate linear algebra allows us to understand antilinear maps while working within $\cvect$, at the cost of some added complexity (pun intended).

\begin{dfn}
  \label{dfn:chom}
  Let $V,W \in \cvect$.  We say that a function $T : V \to W$ is \emph{antilinear} if
  \[  T(\alpha u + \beta v) = \bar{\alpha} T(u) + \bar{\beta} T(v)   \]
  for $\alpha,\beta \in \bbC$ and $u,v \in V$.
  If $T$ is bijective then we say that $T$ is an \emph{anti-isomorphism}; an anti-isomorphism of $V$ with itself will be called an \emph{anti-automorphism}.
  We denote the collection of antilinear maps from $V$ to $W$ by $\cHom (V,W)$.
\end{dfn}

\begin{lem}
  \label{lem:antilinearmaps_vectorspace}
  \begin{enumerate}[(a)]
  \item \label{lem:part:almaps1} For $V, W \in \cvect$, the set $\cHom(V,W)$ is a subspace of the vector space of all functions from $V$ to $W$.
  \item \label{lem:part:almaps2} The composition of two antilinear maps is linear.  The composition of an antilinear map and a linear map, in either order, is antilinear.
  \end{enumerate}
\end{lem}

\begin{proof}
  Straightforward.
\end{proof}

\begin{rem}
  \label{rem:antilinearmaps_vectorspace}
  While the proof of Lemma \ref{lem:antilinearmaps_vectorspace} is trivial, there is some nontrivial content.
  While it is clear how to define addition in $\cHom(V,W)$, there is some choice in the definition of scalar multiplication.  
  If we defined instead $(\lambda T)(v)= \bar{\lambda}T(v)$, this would give us a vector space structure different from that defined in Lemma \ref{lem:antilinearmaps_vectorspace}.
  The key observation is that there is a natural vector space structure on the set of all (not necessarily linear) functions from $V$ to $W$; viewing $\cHom(V,W)$ inside this space allows us to choose naturally between the two options.

  We would like to be able to interpret $\cHom(V,W)$ as a $\Hom$-set in $\cvect$.
  \emph{A priori} this is not possible due to the banal fact that antilinear maps are not linear.
  In \S \ref{sec:functor} we introduce the complex conjugation functor, which is our main technical tool in the rest of this text.
  This will allow us to interpret antilinear maps naturally as linear ones.
\end{rem}

\subsection{The conjugation functor}
\label{sec:functor}

\begin{dfn}
  \label{dfn:conjugation-functor}
  For $V \in \cvect$, we define the \emph{complex conjugate vector space of $V$} or just \emph{complex conjugate of $V$} to be the complex vector space $\bar{V}$ consisting of formal symbols $c_V(v) = \bar{v}$ for $v \in V$ with addition and scalar multiplication given by
  \begin{equation}
    \bar{v} + \bar{w} = \bar{v + w}, \quad \lambda \cdot \bar{v} = \bar{\bar{\lambda} v}
    \label{eq:conj_vectorspace}
  \end{equation}
  for $v,w \in V, \lambda \in \bbC$, respectively.
  Equivalently, we can define the operations in $\bar{V}$ by declaring the map $c_V : V \to \bar{V}$ given by $v \mapsto \bar{v} = c_V(v)$ to be an antilinear bijection.

  For a linear map $T : V \to W$, we define the \emph{complex conjugate of $T$} to be the map $\bar{T} : \bar{V} \to \bar{W}$ given by $\bar{T}(\bar{v}) = \bar{T(v)}$ for $v \in V$, i.e.\ $\bar{T} = c_W \circ T \circ c_V^{-1}$.
  In other words, $\bar{T}$ is the unique map making the diagram
  \begin{equation}
    \label{eq:conj_lin_map_def}
    \begin{CD}
      V  @>T>> W \\
      @V{c_V}VV   @VV{c_W}V \\
      \bar{V} @>>{\bar{T}}> \bar{W}
    \end{CD}
  \end{equation}
  commute.
  It follows from Lemma \ref{lem:antilinearmaps_vectorspace} that $\bar{T}$ is linear.
\end{dfn}

\begin{rem}
  \label{rem:conj_dfn}
  The map $c_V$ is an isomorphism of the underlying real vector spaces of $V$ and $\bar{V}$.  
  Although $V$ and $\bar{V}$ are isomorphic as complex vector spaces since their dimensions are the same, there is in general no natural isomorphism.
  The exception is the ground field $\bbC$, as we will see below in Lemma \ref{lem:cself_conj}.

  We emphasize that $\bar{T}$ is the \emph{unique} linear map making (\ref{eq:conj_lin_map_def}) commute; this follows immediately from the fact that $c_V$ and $c_W$ are bijective.

  Note also that the map $T \mapsto \bar{T}$ is itself an antilinear map since
  \[   \bar{\lambda T} = c_W \circ (\lambda T) \circ c_V^{-1} = \bar{\lambda} \cdot c_W \circ T \circ c_V^{-1} = \bar{\lambda} \cdot \bar{T}.   \]
\end{rem}

\begin{notn}
  We generally use the notation $\bar{v}$ rather than the more cumbersome $c_V(v)$ except in the following two situations when confusion may arise from doing so.
  
  The first possible source of confusion is when $V = \bbC$, where $\bar{\lambda}$ may refer either to the element $c_\bbC(\lambda) \in \bar{\bbC}$ or to the complex conjugate element $\bar{\lambda} \in \bbC$.
  We show in Lemma \ref{lem:cself_conj} that $\bar{\bbC}$ is naturally isomorphic to $\bbC$ and that $c_\bbC(\lambda)$ is identified with $\bar{\lambda}$.

  The second possible source of confusion occurs when discussing linear maps.
  We have seen that we can define the complex conjugate of a linear map $T : V \to W$; the symbol $\bar{T}$ can then refer either to the conjugate map $\bar{T} : \bar{V} \to \bar{W}$ or to the element $c_{\Hom(V,W)}(T)$ of $\bar{\Hom(V,W)}$.
  We show in Proposition \ref{prop:conjugate-hom} that there is a natural way to identify $\bar{T}$ with $c_{\Hom(V,W)}(T)$.
  
  Nevertheless, we will explicitly say what we mean whenever the notation may cause confusion.
\end{notn}

\begin{lem}
  \label{lem:functoriality}
  Complex conjugation is a functor from the category $\cvect$ to itself.  That is, for $V,W,X \in \cvect$ and linear maps $U : V \to W$ and $T : W \to X$ we have
  \begin{enumerate}[(a)]
  \item $\bar{\id_V} = \id_{\bar{V}}$.
  \item $\bar{T \circ U} = \bar{T} \circ \bar{U}$.
  \end{enumerate}
\end{lem}

\begin{proof}
   \begin{enumerate}[(a)]
   \item According to (\ref{eq:conj_lin_map_def}) we have
     \[
     \begin{tikzpicture}[description/.style={fill=white,inner sep=2pt}]
       \matrix (m) [matrix of math nodes, row sep=3em,
       column sep=3em, text height=1.5ex, text depth=0.25ex]
       { V  & V \\
         \bar{V} & \bar{V} \\ };
       \path[->,font=\scriptsize]
       (m-1-1) edge node[auto] {$ \id_V $} (m-1-2)
       edge node[auto,swap] {$ c_V $} (m-2-1)
       (m-1-2) edge node[auto] {$ c_V $} (m-2-2)
       (m-2-1) edge node[auto,swap] {$\bar{\id_V}$} (m-2-2) ;
     \end{tikzpicture}
     \]
     which proves the claim.
   \item Using (\ref{eq:conj_lin_map_def}) twice we have
     \[
     \begin{tikzpicture}[description/.style={fill=white,inner sep=2pt}]
       \matrix (m) [matrix of math nodes, row sep=3em,
       column sep=3em, text height=1.5ex, text depth=0.25ex]
       { V  & W & X \\
         \bar{V} & \bar{W} & \bar{X} \\ };
       \path[->,font=\scriptsize]
       (m-1-1) edge node[auto] {$ U $} (m-1-2)
       (m-1-1) edge node[auto,swap] {$ c_V $} (m-2-1)
       (m-1-2) edge node[auto] {$ T $} (m-1-3)
       (m-1-2) edge node[auto] {$ c_W $} (m-2-2)
       (m-1-3) edge node[auto] {$ c_X $} (m-2-3)
       (m-2-1) edge node[auto,swap] {$\bar{U}$} (m-2-2) 
       (m-2-2) edge node[auto,swap] {$\bar{T}$} (m-2-3) ;
     \end{tikzpicture}
     \]
     so we see that $\bar{T} \circ \bar{U} : \bar{V} \to \bar{X}$ is a linear map which fulfils the uniqueness criterion discussed in Remark \ref{rem:conj_dfn}, and hence $\bar{T} \circ \bar{U} = \bar{T \circ U}$.  \qedhere
  \end{enumerate}
\end{proof}

The next proposition allows us to view antilinear maps as linear ones, as mentioned in Remark \ref{rem:antilinearmaps_vectorspace}.
In order to state the result properly, note that the functor taking a pair of vector spaces $(V,W)$ to $\cHom(V,W)$ is contravariant in $V$ and covariant in $W$: given linear maps $U : V' \to V$ and $R : W \to W'$, the map $\cHom(V,W) \to \cHom(V',W')$ is given by $T \mapsto R \circ T \circ U$, i.e.\
\[ \cHom(U,R) : \left( V \overset{T}{\longrightarrow} W  \right) \mapsto \left( V' \overset{U}{\longrightarrow} V \overset{T}{\longrightarrow}  W \overset{R}{\longrightarrow} W'\right) .  \]
Similarly the functor taking $(V,W)$ to $\Hom(\bar{V}, W)$ is contravariant in $V$ and covariant in $W$: for linear maps $U : V' \to V$ and $R : W \to W'$, the map $\Hom(\bar{V},W) \to \Hom(\bar{V'},W')$ is given by
\[ \Hom(\bar{U},R) : \left( \bar{V} \overset{T}{\longrightarrow} W  \right) \mapsto \left( \bar{V'} \overset{\bar{U}}{\longrightarrow} \bar{V} \overset{T}{\longrightarrow}  W \overset{R}{\longrightarrow} W'\right) .  \]

\begin{prop}
  \label{prop:antilinear_representable}
  For $V,W \in \cvect$, the map 
  \[ \Psi_{VW} : \cHom(V,W) \to \Hom(\bar{V},W) \]
  given by $T \mapsto T \circ c_V^{-1}$ is a linear isomorphism.  If $T$ is an antilinear isomorphism then $\Psi_{VW}(T)$ is a linear isomorphism.  The collection of isomorphisms $(\Psi_{VW})$ is a natural transformation in the sense that for linear maps $U : V' \to V$ and $R : W \to W'$, the following diagram commutes:
  \begin{equation}
    \label{eq:antilinear_representable}
    \begin{CD}
      \cHom(V,W) @>{\cHom(U,R)}>> \cHom(V',W') \\
      @V{\Psi_{VW}}VV @VV{\Psi_{V'W'}}V \\
      \Hom(\bar{V},W) @>>{\Hom(\bar{U},R)}> \Hom(\bar{V'},W)
    \end{CD}
  \end{equation}
\end{prop}

\begin{rem}
  \label{rem:antilinear_representable}
  The definition of $\Psi_{VW}$ is best captured by a diagram.  For an antilinear map $T : V \to W$, the linear map $\Psi_{VW}(T)$ is the unique map making the following diagram commute:
  \[
  \begin{tikzpicture}[description/.style={fill=white,inner sep=2pt}]
    \matrix (m) [matrix of math nodes, row sep=3em,
    column sep=3em, text height=1.5ex, text depth=0.25ex]
    { V  & W \\
      \bar{V} &  \\ };
    \path[->,font=\scriptsize]
    (m-1-1) edge node[auto] {$ T $} (m-1-2)
    edge node[auto,swap] {$ c_V $} (m-2-1)
    (m-2-1) edge node[auto,swap] {$ \Psi_{VW}(T)  $} (m-1-2) ;
  \end{tikzpicture}
  \]
  In fancy terms, this proposition states that for a fixed $V \in \cvect$, the pair $(\bar{V},(\Psi_{V-}))$ represents the covariant functor $\cHom(V, -)$.
\end{rem}

\begin{proof}[Proof of Proposition \ref{prop:antilinear_representable}]
  The image of $\Psi_{VW}$ lands in $\Hom(\bar{V},W)$ since the composition of two antilinear maps is linear.
  The inverse of $\Psi_{VW}$ is given by $T \mapsto T \circ c_V$, so indeed $\Psi_{VW}$ is an isomorphism.
  Since $c_V$ is bijective, we see that $\Psi_{VW}(T)$ is bijective if and only if $T$ is.
  The naturality statement~\eqref{eq:antilinear_representable} follows from commutativity of the diagram
  \[
  \begin{tikzpicture}
    [description/.style={fill=white,inner sep=2pt}]
    \matrix (m) [matrix of math nodes, row sep=3em,
    column sep=3em, text height=1.5ex, text depth=0.25ex]
    { V' & V & W & W' \\
      \bar{V'} & \bar{V} & &  \\ };
    \path[->,font=\scriptsize]
    (m-1-1) edge node[auto] {$ U $} (m-1-2)
    (m-1-1) edge node[auto,swap] {$ c_{V'} $} (m-2-1)
    (m-1-2) edge node[auto] {$ T $} (m-1-3)
    (m-1-2) edge node[auto,swap] {$ c_V $} (m-2-2)
    (m-1-3) edge node[auto] {$ R $} (m-1-4)
    (m-2-1) edge node[auto,swap] {$\bar{U}$} (m-2-2) 
    (m-2-2) edge node[auto,swap] {$ \Psi_{VW}(T) $} (m-1-3) ;    
  \end{tikzpicture} \qedhere
  \]
\end{proof}

\subsection{Further properties of the conjugation functor}
\label{sec:further-conjugation}

In this subsection we examine how the conjugation functor interacts with some natural constructions in $\cvect$.  
We show also that conjugation is an autoequivalence of $\cvect$ and that its square is naturally isomorphic to the identity functor.
We begin by showing that $\bbC$ is naturally isomorphic to its own conjugate.
Our main tool in this subsection will be Proposition \ref{prop:antilinear_representable}.

\begin{lem}
  \label{lem:cself_conj}
  Let $(\lambda \mapsto \bar{\lambda}) \in \cHom (\bbC, \bbC)$ be the usual complex conjugation map from $\bbC$ to itself.
  Then $\gamma = \Psi_{\bbC \bbC}(\lambda \mapsto \bar{\lambda}) \in \Hom(\bar{\bbC}, \bbC)$ is a linear isomorphism.
  Explicitly, $\gamma$ is given by $\gamma(c_\bbC(\lambda)) = \bar{\lambda}$.
\end{lem}

\begin{proof}
  Straightforward.
\end{proof}

The analogue of Lemma \ref{lem:cself_conj} holds for any complex vector space with an antilinear automorphism.
Of course, any vector space has such an anti-automorphism if one fixes a basis; the point is that the anti-automorphism of $\bbC$ is canonical.

\begin{lem}
  \label{lem:conj_involutory}
  The map  $ \sigma_V :  \bar{\bar{V}} \to V$ given by $ \bar{\bar{v}} \mapsto  v$ is a linear isomorphism.  
  Given a linear map $T : V \to W$, the following diagram commutes:
  \begin{equation}
    \label{eq:conj_involutory}
    \begin{CD}
      \bar{\bar{V}}  @>\bar{\bar{T}}>> \bar{\bar{W}} \\
      @V{\sigma_V}VV   @VV{\sigma_W}V \\
      V @>>T> W 
    \end{CD}
  \end{equation}
\end{lem}

\begin{proof}
  Note that $\sigma_V = (c_{\bar{V}} \circ c_V)^{-1} = \Psi_{\bar{V} V}(c_V^{-1})$; since both $c_V$ and $c_{\bar{V}}$ are antilinear isomorphisms, $\sigma_V$ is a linear isomorphism.  
  Commutativity of \eqref{eq:conj_involutory} follows from commutativity of the diagram
\[
  \begin{CD}
    V  @>{c_V}>> \bar{V}  @>{c_{\bar{V}}}>> \bar{\bar{V}} \\
    @V{T}VV @V{\bar{T}}VV @V{\bar{\bar{T}}}VV \\
    W  @>>{c_W}> \bar{W}  @>>{c_{\bar{W}}}> \bar{\bar{W}} 
  \end{CD} 
\] \qedhere
\end{proof}

\begin{prop}
  \label{prop:conjugation-equiv-of-cats}
  The conjugation functor is an autoequivalence of $\cvect$ whose quasi-inverse is itself.  
  The conjugation functor is exact.
\end{prop}

\begin{proof}
  Lemma \ref{lem:conj_involutory} shows that the square of the conjugation functor is naturally isomorphic to the identity functor of $\cvect$.  
  Exactness holds because any equivalence of abelian categories is exact.
\end{proof}

\begin{lem}
  \label{lem:sum_naturality}
  Let $\mathrm{V} = (V_j)_{j \in J}$ be any family of vector spaces.  
  Then the map
  \[ \pi_{\mathrm{V}} : \bar{\bigoplus_{j \in J} V_j} \to \bigoplus_{j \in J} \bar{V_j}  \]
  given by
  \[ \pi_{\mathrm{V}} \left( \bar{(v_j)_{j \in J}} \right) = \left( \bar{v_j}  \right)_{j \in J}   \]
  is a linear isomorphism.
  If $\mathrm{W} = (W_j)_{j \in J}$ is another family of vector spaces indexed by $J$ and if $(T_j : V_j \to W_j)_{j \in J}$ is a family of linear maps, then the diagram
  \begin{equation}
    \begin{CD}
      \label{eq:sum_naturality}
      \bar{\bigoplus_{j \in J}{V_j}}  @>\bar{T_{\mathrm{V}}}>> \bar{\bigoplus_{j \in J} W_j} \\
      @V\pi_{\mathrm{V}}VV @VV\pi_{\mathrm{W}}V \\
      \bigoplus_{j \in J} \bar{V_j} @>>{T_{\bar{\mathrm{V}}}}> \bigoplus_{j \in J} \bar{W_j}
    \end{CD}
  \end{equation}
  commutes, where $T_{\mathrm{V}} = \oplus_{j\in J} T_j$ and $T_{\bar{\mathrm{V}}} = \oplus_{j \in J} \bar{T_j}$.
\end{lem}

\begin{proof}
  Note that $\pi_{\mathrm{V}} = \Psi_{\oplus_j V_j, \oplus_j \bar{V_j}}(\oplus_j c_{V_j})$.  
  Since $\oplus_j c_{V_j}$ is an antilinear isomorphism, it follows from Proposition \ref{prop:antilinear_representable} that $\pi_{\mathrm{V}}$ is a linear isomorphism.
  Commutativity of \eqref{eq:sum_naturality} is straightforward.
\end{proof}

\begin{rem}
  \label{rem:other-constructions}

  We may well ask about the interaction of complex conjugation with several other natural constructions in $\cvect$, namely duals, tensor products, and $\Hom$.
  We defer discussion of these constructions to later sections because of additional complications which are introduced when working with $H$-modules rather than vector spaces.
\end{rem}

\section{Complex conjugation of modules}
\label{sec:conj-of-modules}

In this section $H$ is a fixed Hopf $*$-algebra.
We define complex conjugation of $H$-modules and define the concept of an anti-module map.
We show how complex conjugation interacts with tensor products and duals of modules and we examine the complex conjugates of $\Hom$-spaces also.
We discuss how the results of Section \ref{sec:linalg} extend to $H$-modules.

\subsection{The conjugation functor on $\hmod$}
\label{sec:conj-hmod}

If $H$ is a Hopf algebra over $\bbC$ then $\hmod$ is a subcategory of $\cvect$.  
It is therefore natural to ask whether the complex conjugation functor restricts to an endofunctor on $\hmod$.
This boils down to two questions: first, whether we can define a module structure on $\bar{V}$ for $V \in \hmod$; and second, whether complex conjugates of module maps are again module maps.
The fact that $H$ is a Hopf $*$-algebra allows us to answer these questions affirmatively.

\begin{dfn}
  \label{dfn:conj-module}
  Given any module $V \in \hmod$, we define its \emph{complex conjugate module} to be the complex conjugate vector space $\bar{V}$ with action given by
  \begin{equation}
    \label{eq:conj-module-def}
    a \rhd \bar{v} = \bar{S(a)^* \rhd v}, \text{ or equivalently }  a \rhd c_V(v) = c_V (S(a)^* \rhd v)
  \end{equation}
  for $a \in H$ and $v \in V$.
\end{dfn}

\begin{rem}
  \label{rem:conj-module-def}
  It is straightforward to check that \eqref{eq:conj-module-def} defines an action of $H$ on $\bar{V}$; for the sake of completeness we carry out this computation in the proof of Proposition \ref{prop:conj-module-maps}.
  It is interesting to note that, except for interchanging the order of the antipode and the $*$-operation, this is essentially the only the choice for a module structure on $\bar{V}$: the $*$ is necessary in order to keep the operation of $a$ on $\bar{v}$ linear in $a$, but since the $*$ reverses products, the $S$ is required also to make $\bar{V}$ into a left module.

  If we instead switched the order of the $S$ and the $*$ in \eqref{eq:conj-module-def}, we would have an analogous concept.
  More precisely, we would define $\widetilde{V}$ to be $\bar{V}$ as a complex vector space, with the $H$-action given by $a \rhd \tilde{v} = \widetilde{S(a^*) \rhd v}$ for $a \in A$ and $v \in V$.
  Certainly all of the following theory could be developed in that framework, and while it does not appear that there is a natural transformation which directly connects these two complex conjugation functors, there are some relationships between them.

  First, we claim that $(\widetilde{V})^* \simeq \us \bar{V}$.
  Indeed, note that the underlying complex vector space of each of these modules is just $\Hom(\bar{V},\bbC)$.
  Then for $f\in \Hom(\bar{V},\bbC)$ and $v \in V$, we have 
  \[  (a \rhd f)(\tilde{v}) = f(S(a) \rhd \tilde{v}) = f(\widetilde{S(S(a)^*) \rhd v}) = f(\widetilde{a^* \rhd v}),    \]
  while on the other hand we have
  \[  (a \brhd f)(\bar{v}) = f(S^{-1}(a) \rhd \bar{v}) = f(\bar{a^* \rhd v}),    \]
  so the identity map on the underlying vector space $\Hom(\bar{V},\bbC)$ in fact is an isomorphism of modules $(\widetilde{V})^* \simeq \us \bar{V}$.

  Also we can consider the modules $\widetilde{\bar{V}}$ and $\bar{\widetilde{V}}$.
  For $v \in V$ and $a \in A$, in the former module we have
  \[ a \rhd \widetilde{\bar{v}} = \widetilde{S(a^*) \rhd \bar{v}} = \widetilde{\bar{S^2(a^*)^* \rhd v}} = \widetilde{\bar{S^{-2}(a) \rhd v}},     \]
  while in the latter we get
  \[  a \rhd \bar{\tilde{v}} = \bar{S(a)^* \rhd \tilde{v}} = \bar{\widetilde{S^2(a) \rhd v}}.   \]
  Thus the two actions are related by the automorphism $S^4$ of $H$.
  Hence if $S^4$ is an inner automorphism, the two modules are isomorphic, while if $S^2$ is an inner automorphism then both modules are isomorphic to $V$ (and hence the two conjugation functors are quasi-inverse to one another).
  As mentioned above, if $H$ is quasitriangular then in fact $S^2$ is inner.
  In general, however, this is not the case, although in some situations $S^4$ is ``almost inner'' in a certain sense, and this may allow one to say more about the relationship between the modules $\bar{\widetilde{V}}$ and $\widetilde{\bar{V}}$.
  For finite-dimensional Hopf algebras the result on $S^4$ is due to Radford \cite{Rad76}, extending a result of Larson; for various generalizations to different classes of infinite-dimensional Hopf algebras one can see \cite{BeaBulTor07}, \cite{BeaBul09}, and references therein.
\end{rem}

\begin{prop}
  \label{prop:conj-module-maps}
  Let $V, W \in \hmod$.
  \begin{enumerate}[(a)]
  \item Equation \eqref{eq:conj-module-def} defines an action of $H$ on $\bar{V}$, i.e.~$\bar{V} \in \hmod$.
  \item If $T : V \to W$ is a module map, then $\bar{T} : \bar{V} \to \bar{W}$ is also a module map.
  \item Complex conjugation is an endofunctor of $\hmod$.
  \end{enumerate}
\end{prop}

\begin{proof}
  \begin{enumerate}[(a)]
  \item It is clear that the action is linear in $H$ and in $\bar{v}$.  
    We just check that it is compatible with multiplication in $H$.
    For $a,b \in H$ and $\bar{v} \in \bar{V}$ we have
    \begin{align*}
      a \rhd (b \rhd \bar{v}) & = a \rhd \bar{S(b)^* \rhd v} \\
      & = \bar{ S(a)^* \rhd (S(b)^* \rhd v)} \\
      & = \bar{(S(a)^* S(b)^*) \rhd v} \\
      & = \bar{(S(b)S(a))^* \rhd v} \\
      & = \bar{S(ab)^* \rhd v} \\
      & = (ab) \rhd \bar{v}.
    \end{align*}
  \item 
    For $v \in V$ and $a \in H$ we have
    \begin{align*}
      \bar{T} (a \rhd \bar{v}) & = \bar{T} (\bar{S(a)^* \rhd v}) \\
      & = \bar{ T(S(a)^* \rhd v) } \\
      & = \bar{S(a)^* \rhd T(v)} \\
      & = a \rhd \bar{T(v)} \\
      & = a \rhd (\bar{T}(\bar{v}))
    \end{align*} 
    Thus $\bar{T}$ is a module map.
  \item We know already from Lemma \ref{lem:functoriality} that complex conjugation takes identities to identities and preserves composition.  By part $\mathrm{(b)}$ it takes module maps to module maps, so it defines a functor from $\hmod$ to itself.  \qedhere
  \end{enumerate}
\end{proof}

\subsection{Antimodule maps}
\label{sec:antimodule-maps}

In Section \ref{sec:linalg} we defined antilinear maps and introduced complex conjugation of vector spaces as a way to turn antilinear maps into linear ones.
Here we do the opposite: having defined complex conjugation of modules, we use this to motivate the definition of antimodule maps.
We then show that complex conjugation of modules turns antimodule maps into module maps, as one would hope.
This will prove useful later on when we show how the results of Section \ref{sec:linalg} extend to the category of $H$-modules.

\begin{dfn}
  \label{dfn:antimodule-maps}
  Let $V, W \in \hmod$.  We say that a function $T : V \to W$ is an \emph{antimodule map} if $T$ is antilinear and satisfies
  \[  T(a \rhd v) = S(a)^* \rhd T(v) \]
  for all $a \in H$ and $v \in V$.
  We denote the collection of antimodule maps from $V$ to $W$ by $\cHomH(V,W)$.
\end{dfn}

Of course, the prototype for the definition of an antimodule map is $c_V : V \to \bar{V}$.  The analogue of Lemma \ref{lem:antilinearmaps_vectorspace} $\mathrm{
(b)}$ holds:

\begin{lem}
  \label{lem:antimodulemaps-composition}
  The composition of two antimodule maps is a module map.
  The composition of an antimodule map and a module map, in either order, is an antimodule map.
\end{lem}

\begin{proof}
  Straightforward.
\end{proof}

\begin{prop}
  \label{prop:antimodule-representable}
  Let $V,W \in \hmod$.  Then the linear isomorphism $\Psi_{VW}$ from Proposition \ref{prop:antilinear_representable} restricts to an isomorphism 
  \[  \Psi_{VW} : \cHomH(V,W) \to \homH(\bar{V},W).  \]
  The analogue of the naturality statement \eqref{eq:antilinear_representable} holds with respect to module maps $S : V' \to V$ and $R : W \to W'$.
\end{prop}

\begin{proof}
  Since $c_V$ is an antimodule map, Lemma \ref{lem:antimodulemaps-composition} implies that $\Psi_{VW}$ carries antimodule maps to module maps.  The inverse of $\Psi_{VW}$ is given by $T \mapsto T \circ c_V$.  The naturality statement is immediate from \eqref{eq:antilinear_representable}.
\end{proof}

\begin{prop}
  \label{prop:modulemaps_naturality}
  Let $V$ and $W$ be objects in $\hmod$ and let $\mathrm{V} = (V_j)_{j \in J}$ be a family of objects in $\hmod$.
  \begin{enumerate}[(a)]
  \item The map $\gamma : \bar{\bbC} \to \bbC$ from Lemma \ref{lem:cself_conj} is a module isomorphism.  
  \item The map $\sigma_V : \bar{\bar{V}} \to V$ from Lemma \ref{lem:conj_involutory} is a module isomorphism.
  \item The map $\pi_{\mathrm{V}} : \bar{\bigoplus_{j \in J} V_j} \to \bigoplus_{j \in J} \bar{V_j}$ from Lemma \ref{lem:sum_naturality} is a module isomorphism.
  \end{enumerate}
\end{prop}

\begin{proof}
  We have already shown that $\gamma, \sigma_V$, and $\pi_{\mathrm{V}}$ are linear isomorphisms in Section \ref{sec:linalg}, so it is only left to show that they are morphisms of modules.
  Since these maps are all constructed from antilinear maps using the natural isomorphisms $\Psi$, our strategy is to show that these antilinear maps are actually anti\emph{module} maps; then the result will follow from Proposition \ref{prop:antimodule-representable}. 
  \begin{enumerate}[(a)]
  \item For $a \in H$ and $\lambda \in \bbC$ we have
    \[  \bar{a \rhd \lambda} = \bar{ \counit(a) \lambda} = \bar{\counit(S(a))} \cdot \bar{\lambda} = \counit(S(a)^*) \bar{\lambda} = S(a)^* \rhd \bar{\lambda},   \]
    so $\lambda \mapsto \bar{\lambda}$ is an antimodule map.  
    Hence $\gamma = \Psi_{\bbC \bbC}(\lambda \mapsto \bar{\lambda})$ is a module map.
  \item Since $c_V$ is an antimodule map, $\sigma_V = \Psi_{\bar{V} V} (c_V^{-1})$ is a module map.
  \item Each $c_{V_j}$ is an antimodule map, so $\oplus_{j} c_{V_j}$ is as well.
    Hence $\pi_{\mathrm{V}} = \Psi_{\oplus_j V_j, \oplus_j \bar{V_j}}(\oplus_j c_{V_j})$ is a module map.  \qedhere
  \end{enumerate}
\end{proof}

\subsection{Conjugation of tensor products}
\label{sec:tensor-conj}

In this subsection we show that complex conjugation reverses the order of tensor products of modules.

\begin{prop}
  \label{prop:conj_tensor_prod}
  Given any modules $V,W \in \hmod$, the map $\rho_{VW} : \bar{V \otimes W} \to \bar{W} \otimes \bar{V}$ given by
  $ \bar{v \otimes w} \mapsto \bar{w} \otimes \bar{v}$ is an isomorphism of $H$-modules.  Given module maps $S : W \to Y$ and $T: X \to Z$, the diagram
  \begin{equation}
    \label{eq:conj_tensor_prod}
    \begin{CD}
      \bar{W \otimes X}    @>\bar{S \otimes T}>>      \bar{Y \otimes Z} \\
      @V{\rho_{WX}}VV    @VV{\rho_{YZ}}V \\
      \bar{X} \otimes \bar{W}   @>>{\bar{T} \otimes \bar{S}}>  \bar{Z} \otimes \bar{Y}
    \end{CD}
  \end{equation}
  commutes.  Furthermore, for any $U,V,W$ in $\hmod$, the diagram
  \begin{equation}
    \label{eq:conj_tensor_prod_assoc}
    \begin{CD}
      \bar{U \otimes V \otimes W} @>{\rho_{U,V \otimes W}}>>  \bar{V \otimes W} \otimes \bar{U} \\
      @V{\rho_{U \otimes V, W}}VV   @VV{\rho_{VW} \otimes \id_{\bar{U}}}V \\
      \bar{W} \otimes \bar{U \otimes V} @>>{\id_{\bar{W}} \otimes \rho_{UV}}> \bar{W} \otimes \bar{V} \otimes \bar{U}
    \end{CD}
  \end{equation}
  commutes.
\end{prop}

\begin{proof}
  Our goal is to show that $\eta :  v \otimes w \mapsto \bar{w} \otimes \bar{v}$ is an antimodule map; then it will follow from Proposition \ref{prop:antimodule-representable} that $\rho_{VW} = \Psi_{V \otimes W, \bar{W} \otimes \bar{V}}(\eta)$ is a module map.
  Noting that $\cop(S(a)^*) = S(a_{(2)})^* \otimes S(a_{(1)})^*$, we have
  \begin{align*}
    \eta ( a \rhd v \otimes w) & = \eta (a_{(1)} \rhd v \otimes a_{(2)} \rhd w) \\
    & = \bar{a_{(2)} \rhd w} \otimes \bar{a_{(1)} \rhd v} \\
    & = (S(a_{(2)})^* \rhd \bar{w}) \otimes (S(a_{(1)})^* \rhd \bar{v}) \\
    & = S(a)^* \rhd \eta(v \otimes w),
  \end{align*}
  so $\eta$ is an antimodule map, as claimed. 
  It is clear that $\rho_{VW}$ is a linear isomorphism and that (\ref{eq:conj_tensor_prod}) and (\ref{eq:conj_tensor_prod_assoc}) commute.
\end{proof}

\subsection{Conjugation of duals}
\label{sec:dual-conj}

It is natural to ask whether (left or right) dualization commutes with complex conjugation of modules.
It turns out that the situation is slightly subtler than that.
In fact, complex conjugation intertwines left and right duals, i.e.\ $\bar{V^*} \simeq \us{\bar{V}}$.
This can be seen as follows.
Consider the evaluation pairing $\ev_V : V^* \otimes V \to \bbC$.  
Taking the complex conjugate of this map, composing with the isomorphism $\gamma : \bar{\bbC} \to \bbC$, and precomposing with the isomorphism $\rho_{V^*V}^{-1}$ of $\bar{V} \otimes \bar{V^*}$ with $\bar{V^* \otimes V}$ gives
\[  \gamma \circ \bar{\ev_V} \circ \rho_{V^*V}^{-1} : \bar{V} \otimes \bar{V^*} \to \bbC;   \]
thus $\bar{V^*}$ plays the role of the right dual of $\bar{V}$.
We formalize this as follows:

\begin{prop}
  \label{prop:conjugate-dual}
  Let $V,W \in \hmod$.
  \begin{enumerate}[(a)]
  \item The map $b_V : V^* \to \us{\bar{V}}$ given by
    \[  b_V(f)(\bar{v}) = \bar{f(v)}   \]
    is a bijective antimodule map.
  \item The map $\beta_V : \bar{V^*} \to \us{\bar{V}}$ given by
    \[  \beta_V = \Psi_{V^*, \us{\bar{V}}}(b_V)  \]
    is an isomorphism of modules.
  \item For any module map $T : V \to W$, the following diagram commutes:
    \begin{equation}
      \label{eq:conjugate-dual-naturality}
      \begin{CD}
        \bar{W^*}  @>{\bar{T^{\tr}}}>> \bar{V^*} \\
        @V{\beta_W}VV   @VV{\beta_V}V \\
        \us{\bar{W}}  @>>{\bar{T}^{\tr}}>  \us{\bar{V}}
      \end{CD}
    \end{equation}
  \item The collection of maps $(\beta_V)$ is a natural equivalence of the (contravariant) endofunctors of $\hmod$ given by left dual followed by conjugation, and conjugation followed by right dual, respectively.
  \end{enumerate}
\end{prop}

\begin{proof}
  \begin{enumerate}[(a)]
  \item Note first that $b_V(f) = (\lambda \mapsto \bar{\lambda}) \circ f \circ c_V^{-1}$, where $\lambda \mapsto \bar{\lambda}$ is the complex conjugation map of $\bbC$ to itself.  
    This shows both that $b_V(f)$ is a linear functional on $\bar{V}$ and that $b_V$ itself is an antilinear map.
    The inverse of $b_V$ is given by $g \mapsto (\lambda \mapsto \bar{\lambda}) \circ g \circ c_V$, so $b_V$ is bijective.
    Thus we just need to show that $b_V$ is an antimodule map.
    For $f \in V^*$, $v \in V$, and $a \in H$ we have
    \begin{align*}
      (a \brhd b_V(f)) (\bar{v}) & = b_V(f)(S^{-1}(a) \rhd \bar{v}) \\
      & = b_V(f)(\bar{a^* \rhd v}) \\
      & = \bar{f(a^* \rhd v)} \\
      & = \bar{ (S^{-1}(a^*) \rhd f)(v) } \\
      & = b_V(S(a)^* \rhd f)(\bar{v}),
    \end{align*}
    so $a \brhd b_V(f) = b_V(S(a)^* \rhd f)$.
    Hence $b_V$ is an antimodule map.
  \item Immediate from $\mathrm{(a)}$ together with Proposition \ref{prop:antimodule-representable}.
  \item Straightforward.
  \item Immediate from $\mathrm{(b)}$ and $\mathrm{(c)}$.  \qedhere
  \end{enumerate}
\end{proof}

\subsection{Conjugation of $\Hom$'s}
\label{sec:hom-conj}

The fact that conjugation switches left and right duals implies also that conjugation switches left and right $\Hom$-spaces.
Heuristically, we have
\[ \bar{\homl(V,W)} \simeq \bar{W \otimes V^*} \simeq \bar{V^*} \otimes \bar{W}  \simeq \us{\bar{V}} \otimes \bar{W} \simeq \homr(\bar{V},\bar{W}),  \]
where the first isomorphism comes from the decomposition \eqref{eq:left-hom-decomp}, the second from Proposition \ref{prop:conj_tensor_prod}, the third from Proposition \ref{prop:conjugate-dual}, and the last from \eqref{eq:right-hom-decomp}.
As we noted in Remark \ref{rem:hom-decomp-caveat}, these tensor product decompositions of the $\Hom$-spaces are valid only when at least one of $V$ and $W$ is finite-dimensional.
However, the following results are valid for all modules since they refer only to actions on $\Hom$-spaces and not to the tensor product decompositions.

\begin{prop}
  \label{prop:conjugate-hom}
  Let $V,W \in \hmod$.
  \begin{enumerate}[(a)]
  \item The map $b_V : \homl(V,W) \to \homr(\bar{V},\bar{W})$ given by
    \[  b_V(T)  = \bar{T}  \]
    is a bijective antimodule map.
  \item The map $\beta_V : \bar{\homl(V,W)} \to \homr(\bar{V},\bar{W})$ given by
    \[  \beta_V = \Psi_{\homl(V,W),\homr(\bar{V},\bar{W})}(b_V)   \]
    is an isomorphism of modules.
    On elements $\beta_V$ is given by
    \[  \beta_V(c_{\homl(V,W)}(T)) = \bar{T}.    \]
  \end{enumerate}
\end{prop}

\begin{proof}
  \begin{enumerate}[(a)]
  \item Note first that $b_V(T) = c_W \circ T \circ c_V^{-1}$.  This shows both that $b_V(T)$ is a linear map $\bar{V} \to \bar{W}$ and that $b_V$ itself is an antimodule map.
    The inverse of $b_V$ is given by $U \mapsto c_W^{-1} \circ U \circ c_V$, so $b_V$ is bijective.
    Thus we just need to show that $b_V$ is an antimodule map.
    For $T \in \Hom(V,W)$, $v \in V$, and $a \in H$ we have
    \begin{align*}
      (a \brhd \bar{T}) (\bar{v}) & = a_{(2)} \rhd (\bar{T}(S^{-1}(a_{(1)}) \rhd \bar{v})) \\
      & = a_{(2)} \rhd ( \bar{T}(\bar{a_{(1)}^* \rhd v}  )  ) \\
      & = a_{(2)} \rhd \bar{T (a_{(1)}^* \rhd v)} \\
      & = \bar{S(a_{(2)})^* \rhd T(a_{(1)}^* \rhd v)  } \\
      & = \bar{(S(a)^* \rhd T)(v)} \\
      & = (\bar{S(a)^* \rhd T})(\bar{v}),
    \end{align*}
    so we see that 
    \[  a \brhd b_V(T) = a \brhd \bar{T} =  \bar{S(a)^* \rhd T} = b_V (S(a)^* \rhd T)  \]
    and hence $b_V$ is an antimodule map.
  \item Follows immediately from $\mathrm{(a)}$ together with Proposition \ref{prop:antimodule-representable}.  \qedhere
  \end{enumerate}
\end{proof}

\begin{rem}
  \label{rem:conj-maps-identification}
  Part $\mathrm{(b)}$ of Proposition \ref{prop:conjugate-hom} realizes the identification of the linear map $\bar{T}$ with the abstract element $c_{\Hom(V,W)}(T)$ that we promised in the discussion immediately preceding Lemma \ref{lem:functoriality}. 
\end{rem}

\section{Complex conjugation of algebras}
\label{sec:conj-of-algebras}

In this section we look at algebras in the category $\hmod$ and we show that the complex conjugate of an algebra is again an algebra.
We examine tensor algebras of modules in this light and show that for a module $V$, both $\bar{T(V)}$ and $T(\bar{V})$ satisfy a universal property with respect to lifting of \emph{antimodule} maps from $V$ into algebras in $\hmod$.

For an introduction to module-algebras over a Hopf algebra one can see Section 1.3.3 of \cite{KliSch97} or Section 4.1.C of \cite{ChaPre95}; a much deeper treatment can be found in Chapter 4 of \cite{Mon93}.

\subsection{The complex conjugate of an algebra}
\label{sec:conj_algebra}

Recall that an \emph{$H$-module algebra} is a $\bbC$-algebra $(A,m,u)$, where $A \in \hmod$ and $ m : A \otimes A \to A$ and $u : \bbC \to A$ are morphisms in $\hmod$.  
We will generally just say that $A$ is an algebra in $\hmod$ and refer to the maps $m$ and $u$ explicitly only when necessary; we write the product as $m(x \otimes y) = xy$ and the unit as $u(1_\bbC) = 1_A$.
In terms of elements, the fact that $m$ and $u$ are module maps means that for $a \in H$ and $x,y \in A$ we have
\begin{equation}
  \label{eq:module-alg-conditions}
  a \rhd (xy) = (a_{(1)} \rhd x) (a_{(2)} \rhd y) \quad \text{ and } \quad a \rhd 1_A = \counit(a) 1_A.    
\end{equation}
If $A$ and $B$ are algebras in $\hmod$, then we say that $f : A \to B$ is a \emph{morphism of module-algebras} or a \emph{module-algebra morphism} if $f$ is simultaneously a module map and an algebra homomorphism.

\begin{prop}
  \label{prop:conj-module-alg}
  \begin{enumerate}[(a)]
  \item If $(A,m,u)$ is an algebra in $\hmod$ then $(\bar{A},m_{\bar{A}},u_{\bar{A}})$ is an algebra in $\hmod$ with the structure maps given by 
    \[ u_{\bar{A}} = \bar{u} \circ \gamma^{-1}, \quad m_{\bar{A}} = \bar{m} \circ \rho_{AA}^{-1}, \] 
    where $\gamma : \bar{\bbC} \to \bbC$ and $\rho_{AA} : \bar{A \otimes A} \to \bar{A} \otimes \bar{A}$ are the isomorphisms from Proposition \ref{prop:modulemaps_naturality} $(a)$ and Proposition \ref{prop:conj_tensor_prod}, respectively.
  \item If $f : A \to B$ is a morphism of module-algebras, then $\bar{f} : \bar{A} \to \bar{B}$ is also a morphism of module-algebras.
  \item Complex conjugation is an endofunctor of the category of $H$-module algebras with module-algebra morphisms.
  \end{enumerate}
\end{prop}

\begin{rem}
  \label{rem:conj-module-alg}
  Unwinding the definition of the multiplication in $\bar{A}$, we see that for $a,b \in A$ we have
  \[  \bar{a} \cdot \bar{b} = m_{\bar{A}}(\bar{a} \otimes \bar{b}) = \bar{m}(\bar{b \otimes a}) = \bar{m(b \otimes a)} = \bar{ba}; \]
  this implies that 
  \[ \bar{1_A} \cdot \bar{a} = \bar{a \cdot 1_A} = \bar{a}   \]
  and similarly $\bar{a} \cdot \bar{1_A} = \bar{a}$.
  To summarize, in the conjugate algebra $\bar{A}$ we have
  \begin{equation}
    \label{eq:conj-algebra-operations}
    \bar{a} \cdot \bar{b} = \bar{ba} \quad \text{ and } \quad 1_{\bar{A}} = \bar{1_A}.
  \end{equation}
\end{rem}

\begin{proof}[Proof of Proposition \ref{prop:conj-module-alg}]
  \begin{enumerate}[(a)]
  \item  By Proposition \ref{prop:conj-module-maps}, both $\bar{m}$ and $\bar{u}$ are module maps.
  Then $m_{\bar{A}}$ and $u_{\bar{A}}$ are both module maps, as they are compositions of module maps.
  We just need to verify that $m_{\bar{A}}$ and $u_{\bar{A}}$ satisfy the associativity and unit laws.
  But this is clear from \eqref{eq:conj-algebra-operations}.
  \item  We already know that $\bar{f}$ is a module morphism by Proposition \ref{prop:conj-module-maps}, so we just need to check that it is an algebra homomorphism.  
    For $a,b \in A$ we have
    \begin{align*}
      \bar{f}(\bar{a} \cdot \bar{b}) & = \bar{f}(\bar{ba}) \\
      & = \bar{f(ba)} \\
      & = \bar{f(b) f(a)} \\
      & = \bar{f(a)} \cdot \bar{f(b)} \\
      & = \bar{f}(\bar{a}) \cdot \bar{f}(\bar{b}),
    \end{align*}
    so $\bar{f}$ is multiplicative.
    We have also
    \[ \bar{f}(1_{\bar{A}}) = \bar{f}(\bar{1_A}) = \bar{f(1_A)} = \bar{1_B} = 1_{\bar{B}},   \]
    so $\bar{f}$ is unital, and hence is an algebra homomorphism.
  \item Follows immediately from $\mathrm{(a)}$ and $\mathrm{(b)}$ since we already know that conjugation respects identity maps and composition.  \qedhere
  \end{enumerate}
\end{proof}

We now investigate some properties of morphisms of module-algebras and the corresponding conjugate-linear notions.  
For this we need the following:
\begin{dfn}
  \label{dfn:anti-module-alg-maps}
  If $A$ and $B$ are algebras in $\hmod$ then we say that a function $f : A \to B$ is an \emph{antimodule-algebra morphism} if $f$ is simultaneously an antimodule map, $f(ab) = f(b)f(a)$ for all $a,b \in A$, and $f(1_A) = 1_B$.
\end{dfn}

\begin{lem}
  \label{lem:anti-module-alg-maps}
  Let $A$ and $B$ be algebras in $\hmod$.
  \begin{enumerate}[(a)]
  \item The composition of two antimodule-algebra morphisms is a module-algebra morphism.
    The composition of an antimodule-algebra morphism and a module-algebra morphism, in either order, is an antimodule-algebra morphism.
  \item The map $c_A : A \to \bar{A}$ is an antimodule-algebra morphism.
  \item If $f : A \to B$ is an antimodule map, then $f$ is an antimodule-algebra morphism if and only if $\Psi_{AB}(f) : \bar{A} \to B$ is a module-algebra morphism.
  \end{enumerate}
\end{lem}

\begin{proof}
  \begin{enumerate}[(a)]
  \item Straightforward.
  \item We have $c_A(1_A) = \bar{1_A} = 1_{\bar{A}}$, and for $a,b \in A$ we have
    \[  c_A(ab) = \bar{ab} = \bar{b}\bar{a} = c_A(b)c_A(a).   \]
  \item Since $\Psi_{AB}(f) = f \circ c_A^{-1}$, the result follows immediately from parts $\mathrm{(b)}$ and $\mathrm{(c)}$.
  \end{enumerate}
\end{proof}

\begin{cor}
  \label{cor:c-self-conj-as-algebra}
  The morphism $\gamma : \bar{\bbC} \to \bbC$ from Lemma \ref{lem:cself_conj} is an isomorphism of module-algebras.
\end{cor}

\begin{proof}
  Since $\gamma = \Psi_{\bbC \bbC}(\lambda \mapsto \bar{\lambda})$, this follows from Lemma \ref{lem:anti-module-alg-maps}~$\mathrm{(c)}$ and the observation that $\lambda \mapsto \bar{\lambda}$ is an antimodule-algebra morphism.
\end{proof}

\begin{cor}
  \label{cor:sigma-alg-hom}
  If $A$ is an algebra in $\hmod$ then the map $\sigma_A : \bar{\bar{A}} \to A$ introduced in Lemma~\ref{lem:conj_involutory} is an isomorphism of module-algebras.
\end{cor}

\begin{proof}
  We know that $\sigma_A$ is a linear isomorphism from Lemma \ref{lem:conj_involutory} and that it is a module map by Proposition \ref{prop:modulemaps_naturality}.
  Since $\sigma_A = c_A^{-1} \circ c_{\bar{A}}^{-1}$, we have that $\sigma_A$ is a morphism of module-algebras from parts $\mathrm{(a)}$ and $\mathrm{(b)}$ of Lemma \ref{lem:anti-module-alg-maps}.
\end{proof}

\subsection{The complex conjugate of the tensor algebra}
\label{sec:conj-tensor-alg}

For $V$ in $\hmod$, the tensor algebra of $V$ is defined to be 
\[ T(V)  = \bigoplus_{n=0}^\infty V^{\otimes n}, \]
where by definition $V^{\otimes 0} = \bbC$ (with $H$-action given by the counit), and all of the tensor products are taken over $\bbC$.
We denote by $\iota_V$ the canonical injection of $V$ into $T(V)$, and note that $\iota_V$ is a module map.
We denote homogeneous elements of $T(V)$ by $v_1 \dots v_n$ rather than $v_1 \otimes \dots \otimes v_n$ in order to save space.

Our goal is to show that there is a natural isomorphism of module-algebras between $\bar{T(V)}$ and $T(\bar{V})$.
In order to do so, we begin with a discussion on the universal mapping property of the tensor algebra.
Although the following result is well-known, we include the formal statement and proof for the sake of completeness.

\begin{lem}
  \label{lem:tensor-alg-universal-property}
  Let $V \in \hmod$ and let $T(V)$ be defined as above.
  \begin{enumerate}[(a)]
  \item $T(V)$ is an algebra in $\hmod$ with multiplication given by concatenation of tensors and the unit $\bbC \to T(V)$ given by the isomorphism $\bbC \simeq V^{\otimes 0}$.
  \item For any algebra $A$ in $\hmod$ and any module map $f : V \to A$, there is a unique morphism of module-algebras $\tilde{f} : T(V) \to A$ such that $\tilde{f} \circ \iota_V = f$.
  \end{enumerate}
\end{lem}

\begin{proof}
  \begin{enumerate}[(a)]
  \item The fact that $T(V)$ is an associative algebra over $\bbC$ is straightforward.
    The crux of the matter is to show that the multiplication and unit are maps of $H$-modules.
    For this proof only we denote the multiplication map by $m : T(V) \otimes T(V) \to T(V)$.
    For $v_1, \dots, v_m,w_1, \dots, w_n \in V$, denote $v = v_1 \dots v_m$ and $w = w_1 \dots w_n$.  Then for $a \in H$ we have
    \begin{align*}
      m (a \rhd (v \otimes w)) & = m((a_{(1)} \rhd v)  \otimes  (a_{(2)} \rhd w)) \\
      & = (a_{(1)} \rhd v_1) \dots (a_{(m)} \rhd v_m) (a_{(m+1)} \rhd w_1) \dots (a_{(m+n)} \rhd w_n) \\
      & = a \rhd (v_1 \dots v_m w_1 \dots w_n) \\
      & = a \rhd m(v \otimes w),
    \end{align*}
    so we see that all we really used was coassociativity of the comultiplication.
    The fact that the unit is a module map is immediate, as the module action on $\bbC \simeq V^{\otimes 0}$ is just given by the counit of $H$.
  \item 
    Fix $n \ge 0$.
    If $n = 0$, define $f^0 : V^{\otimes 0} \to A$ by $f^0(1) = 1_A$.
    If $n > 0$, the map $\prod_{j=1}^n V \to A$ given by
    \[  (v_1, \dots, v_n) \mapsto f(v_1) \dots f(v_n)    \]
    is $\bbC$-multilinear and hence induces a unique map $f^n : V^{\otimes n} \to A$ given by
    \[  f^n(v_1 \dots v_n) = f(v_1) \dots f(v_n);  \]
    recall that we are omitting $\otimes$ signs inside $T(V)$.
    We claim that $f^n$ is a map of $H$-modules.
    For $a \in A$, we have
    \begin{align*}
      f^n(a \rhd (v_1 \dots v_n)) & = f^n((a_{(1)} \rhd v_1) \dots (a_{(n)} \rhd v_n)) \\
      & = f(a_{(1)}\rhd v_1) \dots f(a_{(n)} \rhd v_n) \\
      & = (a_{(1)} \rhd f(v_1)) \dots (a_{(n)} \rhd f(v_n)) \\
      & = a \rhd (f(v_1) \dots f(v_n)) \\
      & = a \rhd f^n(v_1 \dots v_n),
    \end{align*}
    which verifies the claim.
    Finally, we define $\tilde{f}$ to be the direct sum of the maps $f^n$ for $n \ge 0$.
    It is immediate from the definition that $\tilde{f}$ is an algebra homomorphism.
    It is also a map of $H$-modules since it is a direct sum of such maps.
    Clearly $\tilde{f} \circ \iota_V = f$, and $\tilde{f}$ is unique with this property since $\iota(V)$ generates $T(V)$ as an algebra.
    \qedhere
  \end{enumerate}
\end{proof}

  We will use similar ideas to show that $T(\bar{V})$ and $\bar{T(V)}$ are isomorphic; both of these algebras are equipped with antimodule maps from $V$ and they satisfy a universal mapping property with respect to algebras with antimodule maps coming from $V$.
  Then the usual abstract nonsense will show that there is a unique isomorphism between $T(\bar{V})$ and $\bar{T(V)}$.

We define maps $\theta_V : V \to T(\bar{V})$ and $\vth_V : V \to \bar{T(V)}$ by
\begin{equation}
  \label{eq:tensor-alg-embedding-maps}
  \theta_V = \iota_{\bar{V}} \circ c_V \quad \text{ and } \quad \vth_V = c_{T(V)} \circ \iota_V,
\end{equation}
where $\iota_V$ and $\iota_{\bar{V}}$ are the embeddings of $V$ and $\bar{V}$ into $T(V)$ and $T(\bar{V})$, respectively, as defined in the discussion at the beginning of the subsection.
Both $\theta_V$ and $\vth_V$ are antimodule maps by Lemma \ref{lem:antimodulemaps-composition}.

\begin{prop}
  \label{prop:conj-tensor-alg-mapping-properties}
  Let $V \in \hmod$ and let $\theta_V$ and $\vth_V$ be as in \eqref{eq:tensor-alg-embedding-maps}.
  Let $A$ be an algebra in $\hmod$ and let $f : V \to A$ be an antimodule map.
  \begin{enumerate}[(a)]
  \item There is a unique module-algebra morphism $\hat{f} : T(\bar{V}) \to A$ such that $\hat{f} \circ \theta_V = f$.
  \item There is a unique module-algebra morphism $\check{f} : \bar{T(V)} \to A$ such that $\check{f} \circ \vth_V = f$.
  \end{enumerate}
\end{prop}

\begin{proof}
  \begin{enumerate}[(a)]
  \item By Proposition \ref{prop:antimodule-representable}, $\Psi_{VA}(f) : \bar{V} \to A$ is a module map.
    Then the universal property of the tensor algebra implies that there is a unique module-algebra morphism $\hat{f} : T(\bar{V}) \to A$ such that $\hat{f} \circ \iota_{\bar{V}} = \Psi_{VA}(f)$.
    Then we have 
    \[  \hat{f} \circ \theta_V = \hat{f} \circ \iota_{\bar{V}} \circ c_V = \Psi_{VA}(f) \circ c_V = f.  \]
    This is encapsulated in the commutativity of the diagram
    
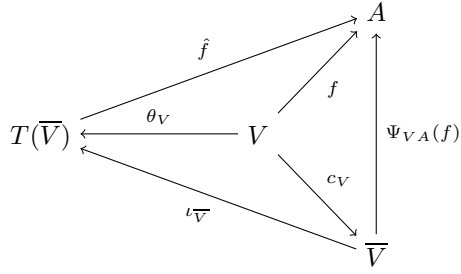
\begin{equation*}
      \begin{tikzpicture}
        [description/.style={fill=white,inner sep=2pt}]
        \matrix (m) [matrix of math nodes, row sep=3em,
        column sep=3em, text height=2ex, text depth=0.25ex]
        { & & & A \\
          T(\bar{V}) & & V & \\
          & & & \bar{V} \\  };
        \path[->,font=\scriptsize]
        (m-2-1) edge node[auto] {$\hat{f}$}  (m-1-4)
        (m-2-3) edge node[auto,swap] {$f$} (m-1-4)
        (m-2-3) edge node[auto,swap] {$\theta_V$} (m-2-1)
        (m-2-3) edge node[auto] {$c_V$} (m-3-4)
        (m-3-4) edge node[auto] {$\iota_{\bar{V}}$} (m-2-1)
        (m-3-4) edge node[auto,swap] {$\Psi_{VA}(f)$} (m-1-4);
      \end{tikzpicture}
    \end{equation*}
    The lower left triangle commutes by definition of $\theta_V$, the right-hand triangle commutes by definition of $\Psi_{VA}(f)$, and the large triangle commutes by the universal property of the tensor algebra.
    Hence the upper left triangle commutes as well, and this shows that $\hat{f}$ has the desired property.

    To see that $\hat{f}$ is unique, note that $\theta_V(V)$ generates $T(\bar{V})$ as an algebra, and the equation $\hat{f} \circ \theta_V = f$ uniquely specifies $\hat{f}$ on the generators.
    Since $\hat{f}$ is an algebra homomorphism, this implies that it is uniquely specified on the algebra $T(\bar{V})$.

  \item By Lemma \ref{lem:antimodulemaps-composition}, the function $c_A \circ f : V \to \bar{A}$ is a module map.
    The universal property of $T(V)$ then gives a unique extension to a module-algebra morphism $f_0 : T(V) \to \bar{A}$ such that $f_0 \circ \iota_V = c_A \circ f$.
    Then $\bar{f_0} : \bar{T(V)} \to \bar{\bar{A}}$ is a morphism of module-algebras by Proposition \ref{prop:conj-module-alg}~$\mathrm{(b)}$.
    We define $\check{f} : \bar{T(V)} \to A$ by $\check{f} = \sigma_A \circ \bar{f_0}$.  
    Since $\sigma_A$ is a module-algebra morphism by Corollary \ref{cor:sigma-alg-hom}, then $\check{f}$ is a module-algebra morphism by Lemma~\ref{lem:anti-module-alg-maps}~$\mathrm{(a)}$. 
    We have
    \begin{align*}
      \check{f} \circ \vth_V & = \sigma_A \circ \bar{f_0} \circ c_{T(V)} \circ \iota_V \\
m      & = \sigma_A \circ c_{\bar{A}} \circ f_0 \circ \iota_V \\
      & = c_{A}^{-1} \circ c_A \circ f \\
      & = f.
    \end{align*}
    This is encapsulated in the commutativity of the diagram
    
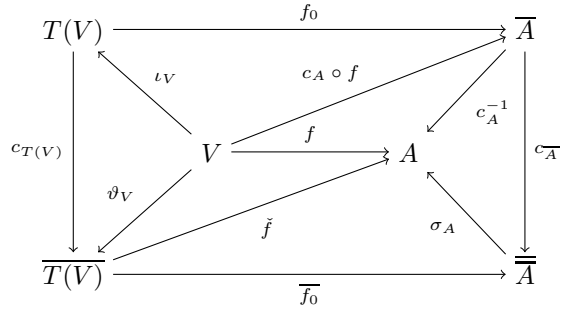
\begin{equation*}
      \label{eq:conj-tensor-alg-diagram2}
      \begin{tikzpicture}
        [description/.style={fill=white,inner sep=2pt}]
        \matrix (m) [matrix of math nodes, row sep=3em,
        column sep=3em, text height=2ex, text depth=0.25ex]
        {T(V) & & & & \bar{A} \\
        & V & & A & \\
        \bar{T(V)} & & & & \bar{\bar{A}} \\ };
        \path[->,font=\scriptsize]
        (m-1-1) edge node[auto] {$f_0$} (m-1-5)
        (m-1-1) edge node[auto,swap] {$c_{T(V)}$} (m-3-1)
        (m-1-5) edge node[auto] {$c_A^{-1}$} (m-2-4)
        (m-1-5) edge node[auto] {$c_{\bar{A}}$} (m-3-5)
        (m-2-2) edge node[auto,swap] {$\iota_V$} (m-1-1)
        (m-2-2) edge node[auto,swap] {$\vth_V$} (m-3-1)
        (m-2-2) edge node[auto] {$c_A \circ f$} (m-1-5)
        (m-2-2) edge node[auto] {$f$} (m-2-4)
        (m-3-5) edge node[auto] {$\sigma_A$} (m-2-4)
        (m-3-1) edge node[auto,swap] {$\check{f}$} (m-2-4)
        (m-3-1) edge node[auto,swap] {$\bar{f_0}$} (m-3-5);
      \end{tikzpicture}
    \end{equation*} 
    The large rectangle commutes by \eqref{eq:conj_lin_map_def}.
    For the five triangles starting on the left-hand side of the diagram and going clockwise, the reasons they commute are, respectively: by definition of $\vth_V$; by the universal property of $T(V)$ (i.e.~ since $f_0$ lifts $c_A \circ f$); trivially; by definition of $\sigma_A$; and by definition of $\check{f}$.
    Then the final interior triangle commutes because everything else does, and this shows that $\check{f}$ has the desired property.

    The proof of uniqueness from part (a) carries over here almost verbatim.  \qedhere
  \end{enumerate}
\end{proof}

Since $T(\bar{V})$ and $\bar{T(V)}$ satisfy the same universal property, they are isomorphic:

\begin{cor}
  \label{cor:conj-tensor-algs-isomorphic}
  There is a unique isomorphism of module-algebras $\kappa_V : T(\bar{V}) \to \bar{T(V)}$ such that $\kappa_V(\bar{v}) = \bar{v}$ for all $v \in V$.
  On simple tensors $\kappa_V$ is given by
  \begin{equation}
    \label{eq:conj-tensor-algs-isomorphism}
    \kappa_V( \bar{v_1} \dots \bar{v_n}) = \bar{v_n \dots v_1},
  \end{equation}
  and hence the restriction of $\kappa_V$ to $\bar{V}^{\otimes n}$ gives an isomorphism $\bar{V}^{\otimes n} \simeq \bar{V^{\otimes n}}$.
\end{cor}

\begin{proof}
  By Proposition \ref{prop:conj-tensor-alg-mapping-properties} $\mathrm{(a)}$ applied to the antimodule map $\vth_V : V \to \bar{T(V)}$, there is a unique map of module-algebras $\kappa_V : T(\bar{V}) \to \bar{T(V)}$ such that $\kappa_V \circ \theta_V = \vth_V$, i.e.~such that $\kappa_V(\bar{v}) = \bar{v}$.

  Applying part $(\mathrm{b})$ of Proposition \ref{prop:conj-tensor-alg-mapping-properties} to the antimodule map $\theta_V : V \to T(\bar{V})$ gives a unique map of module-algebras $\psi_V : \bar{T(V)} \to T(\bar{V})$ such that $\psi_V \circ \vth_V = \theta_V$.

  According to part $\mathrm{(v)}$ of Proposition \ref{prop:conj-tensor-alg-mapping-properties}, there is a unique morphism of module-algebras $\check{\vth}_{V} :  \bar{T(V)} \to \bar{T(V)}$ such that $\check{\vth}_{V} \circ \vth_V = \vth_V$; it is clear that $\check{\vth}_{V} = \id_{\bar{T(V)}}$.
  But on the other hand, we have
  \[ (\kappa_V \circ \psi_V) \circ \vth_V = \kappa_V \circ \theta_V = \vth_V,  \]
  so we must have $\kappa_V \circ \psi_V = \id_{\bar{T(V)}}$ by uniqueness.
  Similarly we have $\psi_V \circ \kappa_V = \id_{T(\bar{V})}$, so both maps are isomorphisms of module-algebras.

  Finally, the formula \eqref{eq:conj-tensor-algs-isomorphism} follows immediately from the description of multiplication in the conjugate of an algebra given in Remark \ref{rem:conj-module-alg}.
\end{proof}

We now use Proposition \ref{prop:conj-tensor-alg-mapping-properties} to show that $T(V)$ has a universal property allowing us to lift antimodule morphisms from $V$ to antimodule-algebra morphisms:

\begin{prop}
  \label{prop:extending-antimodule-maps}
  Let $V \in \hmod$, let $A$ be an algebra in $\hmod$, and let $f : V \to A$ be an antimodule map.
  Then there is a unique antimodule-algebra morphism $f^\sharp : T(V) \to A$ such that $f^\sharp \circ \iota_V = f$.
\end{prop}

\begin{proof}
  By part (b) of Proposition \ref{prop:conj-tensor-alg-mapping-properties} implies that there is a unique homorphism of module-algebras $\check{f} : \bar{T(V)} \to A$ such that $\check{f} \circ \vth_V = f$.
  Then define $f^\sharp$ by $f^\sharp = \check{f} \circ c_{T(V)}$, i.e.~so that the top right triangle of the diagram
  \begin{equation*}
    \label{eq:tensor-alg-antimodule-map-lift}
    \begin{tikzpicture}
      [description/.style={fill=white,inner sep=2pt}]
      \matrix (m) [matrix of math nodes, row sep=3em,
      column sep=3em, text height=2ex, text depth=0.25ex]
      {  T(V) & \bar{T(V)} \\
         V & A \\  };
      \path[->,font=\scriptsize]
      (m-1-1) edge node[auto] {$c_{T(V)}$} (m-1-2)
      (m-1-1) edge node[auto] {$f^\sharp$} (m-2-2)
      (m-1-2) edge node[auto] {$\check{f}$} (m-2-2)
      (m-2-1) edge node[auto] {$\iota_V$} (m-1-1)
      (m-2-1) edge node[auto,swap] {$f$} (m-2-2);
    \end{tikzpicture}
  \end{equation*}
  commutes; recall that $\vth_V = c_{T(V)} \circ \iota_V$.
  The square commutes by the defining property of $\check{f}$, so the lower left triangle commutes as well, i.e.~we have $f^\sharp \circ \iota_V = f$.
  Then $f^\sharp$ is an antimodule-algebra morphism by Lemma \ref{lem:anti-module-alg-maps}, and it is unique since it is uniquely determined on the generators $\iota_V(V)$ of $T(V)$.
\end{proof}

\section{$*$-structures on modules}
\label{sec:star-modules}

We now turn to the main theme of this chapter, which is $*$-structures.
Although our goal is to formulate the correct notion of a $*$-algebra in the category $\hmod$, it is useful to consider modules first and then extend to algebras afterward.
In \cref{sec:star-modules-definition} we begin by defining $*$-modules and $*$-morphisms.
We show in \cref{prop:star-structs-basic-properties} that \(*\)-modules are the ``fixed points up to homotopy'' of the conjugation functor in the sense that the conjugation functor is naturally isomorphic to the identity functor when restricted to the subcategory of \(\hmod\) consisting of \(*\)-modules.
We then show that modules of the form $\bar{V} \otimes V$ and $V \otimes \bar{V}$ carry natural $*$-structures; these $*$-modules will be used later in Section \ref{sec:inner-products} to formulate the notion of an inner product in our framework.
We also discuss $*$-submodules and quotients.

\subsection{The category $\hmods$}
\label{sec:star-modules-definition}

The usual definition of a $*$-structure on, say, a complex algebra $A$ is an antilinear map $* : A \to A$ such that $(a^*)^* = a$ and $(ab)^* = b^*a^*$.
Simply omitting the last condition then gives a reasonable definition of a $*$-structure on a vector space.
Translating this into our framework of complex conjugate modules using the map $\Psi_{VV}$ leads to the following definition:

\begin{dfn}
  \label{dfn:star-module}
  We say that a module $V \in \hmod$ is a \emph{$*$-module} if $V$ is equipped with a module map $* : \bar{V} \to V$, such that $* \circ \bar{*} = \sigma_{VV} : \bar{\bar{V}} \to V$.
  We refer to the map $*$ as a \emph{$*$-structure on $V$}.
  
  If $W \in \hmod$ is also a $*$-module, then we say that a linear map $T : V \to W$ is a \emph{$*$-morphism} or a \emph{$*$-map} if $T \circ * = * \circ \bar{T}$, i.e.~if the following diagram commutes:
  \begin{equation}
    \label{eq:star-morphism-dfn}
    \begin{CD}
      \bar{V} @>{*}>> V \\
      @V{\bar{T}}VV @VV{T}V \\
      \bar{W} @>>{*}> W
    \end{CD}
  \end{equation}
  If $T$ is also a module map then we may refer to it as a \emph{$*$-module morphism} or \emph{$*$-module map}.
  If $v \in V$ and $\bar{v}^* = v$ then we will say that $v$ is a \emph{self-adjoint element} or just that $v$ is \emph{self-adjoint}.
\end{dfn}

\begin{rem}
  \label{rem:star-struct-dfn}
  It might seem more natural for a $*$-structure on $V$ to be a map originating from $V$ rather than from $\bar{V}$.
  However, since a $*$-structure in the usual formulation is an antilinear map from $V$ to $V$, Proposition \ref{prop:antimodule-representable} tells us that Definition \ref{dfn:star-module} is appropriate.

  It follows immediately from the requirement $* \circ \bar{*} = \sigma_V$ that the map $*$ itself must be an isomorphism of modules, hence we do not require it in the definition.
\end{rem}

\begin{lem}
  \label{lem:star-modules-are-a-category}
  \begin{enumerate}[(a)]
  \item For any $*$-module $V$, $\id_V$ is a morphism of $*$-modules.
  \item The composition of two $*$-module morphisms is a $*$-module morphism.
  \end{enumerate}
\end{lem}

\begin{proof}
  Straightforward.
\end{proof}

\begin{dfn}
  \label{dfn:category-of-star-modules}
  The preceding lemma tells us that the collection of all $*$-modules in $\hmod$ together with all $*$-module maps forms a subcategory of $\hmod$.
  We denote this category by $\hmods$.
\end{dfn}

\begin{notn}
  We will mostly not need to work with more than one $*$-structure on any given module at a time, so we generally employ the notation $*$ for all $*$-structures we consider.
  If necessary we will decorate the $*$'s with labels.

  If $V \in \hmods$, we write $\bar{v}^*$ rather than $*(\bar{v})$.
  In this notation, the requirement that $* \circ \bar{*} = \sigma_V$ means that $((\bar{\bar{v}})^{\bar{*}})^* = v$ for all $v \in V$.
  It is quite cumbersome to carry around the $\bar{*}$ in the superscript, and henceforth we will just write $*$ for the morphism $\bar{V} \to V$ as well as for its complex conjugate $\bar{\bar{V}} \to \bar{V}$.
  Thus the condition of involutivity becomes the slightly more palatable $(({\bar{\bar{v}}})^*)^* = v$; omitting parentheses gives the almost satisfactory $\bar{\bar{v}}^{**} = v$.
\end{notn}

We now show that out definition of a $*$-structure is equivalent to the usual definition in terms of an involutive antilinear (antimodule) map:

\begin{prop}
  \label{prop:constructing-stars}
  Let $V,W \in \hmod$.
  \begin{enumerate}[(a)]
  \item If $\dagger : v \mapsto v^\dagger$ is an antimodule map $V \to V$, then $* = \Psi_{VV}(\dagger) : \bar{V} \to V$ is a $*$-structure in the sense of Definition \ref{dfn:star-module} if and only if $(v^\dagger)^\dagger = v$ for all $v \in V$.
  \item Let $\dagger : V \to V$ and $\dagger : W \to W$ be involutive antimodule maps with corresponding $*$-structures as in (a).  
    Then a linear map $T : V \to W$ is a $*$-map if and only if $T(v^\dagger) = (Tv)^\dagger$ for all $v \in V$.
  \end{enumerate}
\end{prop}

\begin{proof}
  \begin{enumerate}[(a)]
  \item From Proposition \ref{prop:antimodule-representable} we know that $*$ is a module map.
    For $v \in V$ we have 
    \[  \bar{v}^* = \Psi_{VV}(\dagger)(\bar{v}) = v^\dagger, \]
    so $\bar{\bar{v}}^* = \bar{v^\dagger}$.
    Hence
    \[ \bar{\bar{v}}^{**} = (\bar{v^\dagger})^* = (v^\dagger)^\dagger,    \]
    so we see that $\bar{\bar{v}}^{**} = \sigma_V(\bar{\bar{v}}) = v$ if and only if $(v^\dagger)^\dagger = v$.
  \item On the one hand we have
    \[  T(\bar{v}^*) = T(v^\dagger),   \]
    while on the other hand we have
    \[  \left( \bar{T} (\bar{v}) \right)^* = \left( \bar{Tv} \right)^* = (Tv)^\dagger,   \]
    so we see that the diagram \eqref{eq:star-morphism-dfn} commutes if and only if $T(v^\dagger) = (Tv)^\dagger$.  \qedhere
  \end{enumerate}
\end{proof}

\begin{rem}
  \label{rem:star-structs-equiv}
  The upshot of Proposition \ref{prop:constructing-stars} is that we can work with involutive antimodule maps rather than directly with Definition \ref{dfn:star-module}.
  This is convenient, especially in the proofs below, where this approach allows us to avoid the complication of carrying $\bar{\phantom{=}}$'s around all over the place.
\end{rem}

\begin{prop}
  \label{prop:star-structs-basic-properties}
  Let $V,W \in \hmods$.
  \begin{enumerate}[(a)]
  \item The module $\bar{V}$ is a $*$-module with the $*$-structure $\bar{*} : \bar{\bar{V}} \to \bar{V}$.
  \item If $T : V \to W$ is a morphism of $*$-modules then $\bar{T} : \bar{V} \to \bar{W}$ is a morphism of $*$-modules with respect to the $*$-structures defined in $\mathrm{(a)}$.
  \item Complex conjugation is an endofunctor of $\hmods$, and moreover the morphisms $* : \bar{V} \to V$ give a natural isomorphism of the complex conjugation functor with the identity functor of $\hmods$.
  \end{enumerate}
\end{prop}

\begin{proof}
  \begin{enumerate}[(a)]
  \item The complex conjugate morphism $\bar{*} : \bar{\bar{V}} \to \bar{V}$ is a module morphism by Proposition \ref{prop:conj-module-maps}.
    Then we have (trivially and unenlighteningly)
    \[  \bar{*} \circ \bar{\bar{*}} = \bar{* \circ \bar{*}} = \bar{\sigma_V} = \sigma_{\bar{V}},   \]
    so $\bar{*}$ is a $*$-structure on $\bar{V}$.
  \item Follows immediately from taking the complex conjugate of the diagram \eqref{eq:star-morphism-dfn}.
  \item Since each morphism $* : \bar{V} \to V$ is an isomorphism, the statement is exactly the commutativity of the diagram \eqref{eq:star-morphism-dfn}.  \qedhere
  \end{enumerate}
\end{proof}

\subsection{Building up the category $\hmods$}
\label{sec:building-up-hmods}

Although we have defined the category $\hmods$, we have not exhibited any modules in it.
In this subsection we show that $\bbC$ is a $*$-module and that we can in fact construct a large class of $*$-modules without knowing anything about the specific Hopf $*$-algebra $H$ that we are working with.
We also discuss direct sums and tensor products of $*$-modules.

\begin{lem}
  \label{lem:c-star-module}
  The map $\gamma : \bar{\bbC} \to \bbC$ is a $*$-structure on $\bbC$.
\end{lem}

\begin{proof}
  Follows immediately from Proposition \ref{prop:constructing-stars}.
\end{proof}

\begin{prop}
  \label{prop:enveloping-modules}
  Let $V \in \hmod$.
  \begin{enumerate}[(a)]
  \item The module $V^e = \bar{V} \otimes V$ is a $*$-module with $*$-structure given by
    \[   (\bar{\bar{v} \otimes w})^* =  \bar{w} \otimes v    \]
    for $v,w \in V$.
  \item The module $\ue{V} = V \otimes \bar{V}$ is a $*$-module with $*$-structure given by
    \[   (\bar{v \otimes \bar{w}})^* =  w \otimes \bar{v}    \]
    for $v,w \in V$.
  \end{enumerate}
\end{prop}

\begin{proof}
  We show the proof for $\mathrm{(a)}$ only; $\mathrm{(b)}$ is similar.  
  Let us define a map $\dagger : V^e \to V^e$ by
  \[  (\bar{v} \otimes w)^\dagger = \bar{w} \otimes v.   \]
  It is clear that $(x^\dagger)^\dagger = x$ for all $x \in V^e$.
  Moreover we claim that $\dagger$ is an antimodule map.
  For $v,w \in V$ and $a \in H$ we have
  \begin{align*}
    (a \rhd (\bar{v} \otimes w))^\dagger & = [(a_{(1)} \rhd \bar{v}) \otimes (a_{(2)} \rhd w)]^\dagger \\
    & = [( \bar{S(a_{(1)})^* \rhd v  }) \otimes (a_{(2)} \rhd w  )]^\dagger \\
    & = (\bar{a_{(2)} \rhd w}) \otimes (S(a_{(1)})^* \rhd v) \\
    & = (S(a_{(2)})^* \rhd \bar{w}) \otimes (S(a_{(1)})^* \rhd v) \\
    & = S(a)^* \rhd (\bar{w} \otimes v) \\
    & = S(a)^* \rhd (\bar{v} \otimes w)^\dagger,
  \end{align*}
  which verifies the claim.
  Then Proposition \ref{prop:constructing-stars} implies that $* = \Psi_{V^e V^e}(\dagger)$ is a $*$-structure.
\end{proof}

\begin{prop}
  \label{prop:direct-sum-star-modules}
  For a family $(V_j)_{j \in J}$ in $\hmods$ with $*$-structures $* : \bar{V_j} \to V_j$, the direct sum $\bigoplus_{j \in J} V_j$ is a $*$-module with the $*$-structure given by
  \begin{equation}
    \label{eq:direct-sum-star-structure}
    (\bar{(v_j)_{j \in J}})^* = (\bar{v_j}^{*})_{j \in J}.   
  \end{equation}
\end{prop}

\begin{proof}
  For each $j \in J$, define a map $\dagger : V_j \to V_j$ by $\dagger = \Psi_{V_j V_j}^{-1}(*)$.
  Each $\dagger$ is an antimodule map by Proposition \ref{prop:antimodule-representable}, and $({v_j}^\dagger)^\dagger = v_j$ for all $v_j \in V_j$ by Proposition \ref{prop:constructing-stars}.
  
  Then we define $\dagger : \bigoplus_{j \in J} V_j \to \bigoplus_{j \in J} V_j$ to be the direct sum of the maps $\dagger : V_j \to V_j$ and observe that $\dagger$ is an involutive antimodule map since all of its direct summands are such.
  Finally, the map \eqref{eq:direct-sum-star-structure} is given by
  \[   * = \Psi_{\oplus_j V_j, \oplus_j V_j}(\dagger),    \]
  so it is a $*$-structure by Proposition \ref{prop:constructing-stars}.
\end{proof}

\begin{prop}
  \label{prop:starmods-tensor-powers}
  For $V \in \hmods$ and $n$ a positive integer, the module $V^{\otimes n}$ is a $*$-module with the $*$-structure given by
  \begin{equation}
    \label{eq:tensor-power-star-struct}
    (\bar{v_1 \otimes \dots \otimes v_n})^* = \bar{v_n}^* \otimes \dots \otimes \bar{v_1}^*.   
  \end{equation}
\end{prop}

\begin{proof}
  Let $\dagger = \Psi_{VV}^{-1}(*) : V \to V$ be the involutive antimodule map corresponding to the given $*$-structure, and then define another map $\dagger : V^{\otimes n} \to V^{\otimes n}$ by
  \[ (v_1 \otimes \dots \otimes v_n)^\dagger = v_n^\dagger \otimes \dots \otimes v_1^\dagger.   \]
  This map is involutive by construction, and we claim that it is an antimodule map as well.
  For $a \in H$ we have
  \begin{align*}
    [a \rhd (v_1 \otimes \dots \otimes v_n)]^\dagger & = [(a_{(1)} \rhd v_1) \otimes \dots \otimes (a_{(n)} \rhd v_n)]^\dagger \\
    & = (a_{(n)} \rhd v_n)^\dagger \otimes \dots \otimes (a_{(1)} \rhd v_1)^\dagger \\
    & = (S(a_{(n)})^* \rhd v_n^\dagger) \otimes \dots \otimes (S(a_{(1)})^* \rhd v_1^\dagger) \\
    & = S(a)^* \rhd (v_n^\dagger \otimes \dots \otimes v_1^\dagger) \\
    & = S(a)^* \rhd (v_1 \otimes \dots \otimes v_n)^\dagger,
  \end{align*}
  so the claim is verified.
  Finally, the map \eqref{eq:tensor-power-star-struct} is given by $* = \Psi_{V^{\otimes n} V^{\otimes n}}(\dagger)$, so it is a $*$-structure by Proposition \ref{prop:constructing-stars}.
\end{proof}

\begin{rem}
  \label{rem:no-arbitrary-tensor-prods}
  It seems that in general there is no way to define a $*$-structure on the tensor product of two arbitrary $*$-modules $V$ and $W$.
  In order to do so we would need a module map $* : \bar{V \otimes W} \to V \otimes W$; since $\bar{V \otimes W} \simeq \bar{W} \otimes \bar{V}$, however, the tensor product of the $*$-structures gives us a map $* : \bar{V \otimes W} \to W \otimes V$.
  In some situations there is a braiding on the category $\hmod$ (or $\hmodf$), and if there is some compatibility between the braidings and the $*$-structures, then more can be said.
  We will revisit this issue later in Section \ref{sec:stars-and-braidings}.

  It seems also that there is no natural way to define a \(*\)-structure on the dual of a \(*\)-module.
  Indeed, if \(* : \bar{V} \to V\) is a \(*\)-structure, then the dual map is \(*^{\tr} : V^* \to \bar{V}^* \cong \bar{\us V}\), and taking the complex conjugate of this gives a map \(\bar{V^*} \to \us V\), which is not of the correct form to be a \(*\)-structure.
\end{rem}

\subsection{$*$-submodules and quotients}
\label{sec:star-submodules}

We now prove the unsurprising fact that the quotient of a $*$-module $V$ by a submodule $W$ inherits a $*$-structure when $W$ is stable under the $*$-operation, and we show also that the kernel and image of a $*$-module morphism are $*$-submodules.

\begin{dfn}
  \label{dfn:star-submodule}
  Let $V \in \hmods$ and let $W \subseteq V$ be a submodule.
  We say that $W$ is a \emph{$*$-submodule} if $\bar{W}^* \subseteq W$.
  (Note that involutivity of the $*$-structure on $V$ implies that $\bar{W}^* = W$ if $W$ is a $*$-submodule.)
\end{dfn}

\begin{lem}
  \label{lem:quotient-by-star-submodule}
  Let $V \in \hmods$ and let $W$ be a $*$-submodule.
  Then there is a unique $*$-structure on the quotient $V/W$ such that the quotient map $q : V \to V/W$ is a morphism of $*$-modules, i.e.~so that the following diagram commutes:
  \begin{equation}
    \label{eq:quotient-by-star-submodule-diagram}
    \begin{CD}
      \bar{V} @>{*}>> V \\
      @V{\bar{q}}VV @VV{q}V \\
      \bar{V/W} @>>{*}> V/W
    \end{CD}
  \end{equation}
\end{lem}

\begin{proof}
  Let $\dagger : V \to V$ be the involutive antimodule map corresponding to $*$ by Proposition \ref{prop:constructing-stars}.
  The fact that $W$ is a $*$-submodule implies that $W^\dagger \subseteq W$.
  Then $(v + W)^\dagger = v^\dagger + W$ is a well-defined involutive antimodule map on $V/W$, so Proposition \ref{prop:constructing-stars} (a) implies that the corresponding module map is a $*$-structure.
  By definition we have $q(v)^\dagger = q(v^\dagger)$, so $q$ is a $*$-map by part (b) of Proposition \ref{prop:constructing-stars}.
\end{proof}

\begin{lem}
  \label{lem:kernels-and-images-are-star-submodules}
  Let $V,W \in \hmods$ and let $T : V \to W$ be a $*$-module map.
  \begin{enumerate}[(a)]
  \item $\ker(T)$ is a $*$-submodule of $V$.
  \item $\im(T)$ is a $*$-submodule of $W$.
  \end{enumerate}
\end{lem}

\begin{proof}
  \begin{enumerate}[(a)]
  \item For $v \in \ker(T)$, we have
    \[  T(\bar{v}^*) = \left( \bar{T} (\bar{v})  \right)^* = \left( \bar{Tv} \right)^* = 0,    \]
    so $\bar{v}^* \in \ker(T)$, and hence $\ker(T)$ is a $*$-submodule of V.
  \item For $w \in \im(T)$, write $w = T(v)$ for some $v \in V$.
    Then we have
    \[  \bar{w}^* =  \left( \bar{Tv} \right)^* = T(\bar{v}^*) \in \im(T),    \]
    so $\im(T)$ is a $*$-submodule of $W$.  \qedhere
  \end{enumerate}
\end{proof}

\section{$*$-structures on algebras}
\label{sec:star-algebras}

In this section we extend to algebras our discussion of $*$-structures on modules from Section \ref{sec:star-modules}.
We define $*$-algebras and their morphisms and show how to construct $*$-algebra structures from involutive antimodule-algebra maps.
We discuss $*$-ideals and quotients, and give a criterion for when an ideal generated by a submodule of a $*$-algebra is a $*$-ideal.
We show that $*$-structures on a module can be lifted to a $*$-structure on its tensor algebra, and further that the the tensor algebra has a universal mapping property for $*$-module morphisms.
We also make some observations about $*$-structures on algebras presented by generators and relations.

\subsection{Algebras in $\hmods$}
\label{sec:star-algebras-def}

\begin{dfn}
  \label{dfn:star-algebra}
  Suppose that $A$ is an algebra in $\hmod$ and that $A$ is equipped with a $*$-structure in the sense of Definition \ref{dfn:star-module}.
  Then we say that $A$ is a \emph{$*$-module algebra}, or just a \emph{$*$-algebra in $\hmod$} if $* : \bar{A} \to A$ is an algebra homomorphism, and we refer to $*$ as a \emph{$*$-algebra structure on $A$}.

  If $A$ and $B$ are $*$-algebras in $\hmod$, then a \emph{$*$-homomorphism} is an algebra homomorphism $f : A \to B$ which is also a map of $*$-modules.
\end{dfn}

\begin{prop}
  \label{prop:constructing-star-algebras}
  Suppose that $A$ is an algebra in $\hmod$ and $\dagger : A \to A$ is an antimodule map.
  Then $* = \Psi_{AA}(\dagger)$ is a $*$-algebra structure on $A$ if and only if $\dagger$ is an antimodule-algebra homomorphism satisfying $(a^\dagger)^\dagger = a$ for all $a \in A$. 
\end{prop}

\begin{proof}
  Proposition \ref{prop:constructing-stars} (a) implies that $*$ is a $*$-module structure on $A$ if and only if $\dagger$ is an involutive antimodule map.
  Then Lemma \ref{lem:anti-module-alg-maps} (c) implies that $*$ is a module-algebra morphism if and only if $\dagger$ is an antimodule-algebra morphism.
\end{proof}

\subsection{$*$-ideals and quotients}
\label{sec:star-ideals}

Here we show that, as one might expect, the quotient of a $*$-algebra by an ideal inherits a $*$-structure exactly when the ideal is a $*$-algebra.
This will be useful in \S \ref{sec:tensor-alg-star-structs} when we discuss $*$-structures on algebras defined by generators and relations.

\begin{dfn}
  \label{dfn:star-ideal}
  An ideal (two-sided) $I$ in a $*$-algebra $A$ in $\hmod$ is called a \emph{$*$-ideal} if $(\bar{I})^* \subseteq I$.
  (Note that this implies that $(\bar{I})^* = I$ because of involutivity.)
\end{dfn}

\begin{lem}
  \label{lem:quotient-by-star-ideal}
  Let $A$ be a $*$-algebra in $\hmod$ and suppose that the ideal $I$ is a $*$-submodule of $A$ in the sense of Definition \ref{dfn:star-submodule}.
  With respect to the $*$-structure on $A/I$ defined in Lemma \ref{lem:quotient-by-star-submodule}, the quotient map $q : A \to A/I$ is a morphism of $*$-module algebras.
\end{lem}

\begin{proof}
  The quotient map is a map of module-algebras by construction, and it is a $*$-map by Lemma \ref{lem:quotient-by-star-submodule}.
\end{proof}

\begin{lem}
  \label{lem:generating-star-ideals}
  Let $A$ be a $*$-algebra in $\hmod$ and let $W \subseteq A$ be a $*$-submodule.
  Then the two-sided ideal $J_W$ of $A$ generated by $W$ is a $*$-submodule of $A$, and hence is a $*$-ideal.
\end{lem}

\begin{proof}
  The elements of $J_W$ are finite sums of terms of the form $awb$, where $a,b \in A$ and $w \in W$.
  For $x \in H$ we have
  \[  x \rhd (awb) = (x_{(1)} \rhd a)(x_{(2)} \rhd w)(x_{(3)} \rhd b);   \]
  since $W$ is a submodule, the middle term $x_{(2)} \rhd w$ is in $W$, so $x \rhd (awb)$ is in $J_W$.
  Hence $J_W$ is a submodule.

  To see that $J_W$ is a $*$-ideal, we need to show that $(\bar{awb})^* \in J_W$ for all $a,b \in A$ and $w \in W$.
  We have
  \[ (\bar{awb})^* = (\bar{b} \bar{w} \, \bar{a})^* =  \bar{b}^* \bar{w}^* \bar{a}^*; \]
  since $W$ is a $*$-submodule, the middle term $\bar{w}^*$ is in $W$, so $(\bar{awb})^*$ is in $J_W$.
  Thus $J_W$ is a $*$-ideal in $A$.
\end{proof}

\subsection{$*$-structures on $T(V)$}
\label{sec:tensor-alg-star-structs}

Here we show that a $*$-structure on a module $V$ extends uniquely to the tensor algebra $T(V)$, and we show that $T(V)$ has a universal mapping property for $*$-algebras equipped with module maps coming from $V$.

\begin{prop}
  \label{prop:tensor-alg-star-struct}
  Let $V \in \hmods$ and let $T(V)$ be the tensor algebra of $V$.
  There is a unique $*$-module algebra structure on $T(V)$ such that the inclusion map $\iota_V : V \to T(V)$ is a $*$-module morphism, i.e.~such that the following diagram commutes:
  \begin{equation}
    \label{eq:tensor-alg-star-struct}
    \begin{CD}
      \bar{V} @>{*}>> V \\
      @V{\bar{\iota_V}}VV @VV{\iota_V}V \\
      \bar{T(V)} @>>{*}> T(V)
    \end{CD}
  \end{equation}
  For an element $v_1 \otimes \dots \otimes v_n \in V^{\otimes n}$, we have
  \begin{equation}
    \label{eq:tensor-alg-star-struct-2}
    (\bar{v_1 \otimes \dots \otimes v_n})^* = \bar{v_n}^* \otimes \dots \otimes \bar{v_1}^*.
  \end{equation}
  Hence the $*$-structure on $T(V)$ coincides with the direct sum of the $*$-structures on the modules $V^{\otimes n}$ coming from Proposition \ref{prop:starmods-tensor-powers}.
\end{prop}

\begin{proof}
  Let $\dagger : V \to V$ be the involutive antimodule map corresponding to the given $*$-structure on $V$.
  Now $\iota_V \circ \dagger$ is an antimodule map $V \to T(V)$, so by Proposition \ref{prop:extending-antimodule-maps} it extends uniquely to an antimodule-algebra morphism, which we also denote $\dagger : T(V) \to T(V)$.
  Note that $\dagger \circ \dagger$ is a module-algebra endomorphism of $T(V)$ which is the identity on the generators, so it must be the identity map.
  Hence Proposition \ref{prop:constructing-stars} (a) implies that $* = \Psi_{T(V) T(V)}(\dagger)$ is a $*$-structure on $T(V)$.

  The fact that $\iota_V$ is a $*$-map follows from commutativity of the diagram
      \begin{equation*}
      \label{eq:tensor-alg-star-struct-bigdiagram}
      \begin{tikzpicture}
        [description/.style={fill=white,inner sep=2pt}]
        \matrix (m) [matrix of math nodes, row sep=3em,
        column sep=3em, text height=2ex, text depth=0.25ex]
        { \bar{V} & & & \bar{T(V)}  \\
          & V & T(V) & \\
          V & & & T(V) \\ };
        \path[->,font=\scriptsize]
        (m-1-1) edge node[auto] {$\bar{\iota_V}$} (m-1-4)
        (m-1-1) edge node[auto,swap] {$*$} (m-3-1)
        (m-1-1) edge node[auto] {$c_V^{-1}$} (m-2-2)
        (m-1-4) edge node[auto,swap] {$c_{T(V)}^{-1}$} (m-2-3)
        (m-1-4) edge node[auto] {$*$} (m-3-4)
        (m-2-2) edge node[auto,swap] {$\iota_V$} (m-2-3)
        (m-2-2) edge node[auto] {$\dagger$} (m-3-1)
        (m-2-3) edge node[auto,swap] {$\dagger$} (m-3-4)
        (m-3-1) edge node[auto,swap] {$\iota_V$} (m-3-4);
      \end{tikzpicture}
    \end{equation*} 
    The upper trapezoid commutes by definition of the complex conjugate of $\iota_V$.
    The triangle on the right commutes by the definition of $*$.
    The lower trapezoid commutes by Proposition \ref{prop:extending-antimodule-maps}.
    The triangle on the left commutes by definition of $\dagger$.
    Thus the large rectangle commutes as well, which gives \eqref{eq:tensor-alg-star-struct}.

    The $*$-structure on $T(V)$ is unique with this property since $T(V)$ is generated as an algebra by $\iota_V(V)$, and \eqref{eq:tensor-alg-star-struct} uniquely specifies $*$ on the generators.

    The equation \eqref{eq:tensor-alg-star-struct-2} follows from the fact that $*$ is an algebra map together with the description \eqref{eq:conj-algebra-operations} of multiplication in the complex conjugate of an algebra.
    We see that \eqref{eq:tensor-alg-star-struct-2} is identical to \eqref{eq:tensor-power-star-struct}, so this $*$-structure on $T(V)$ is indeed the direct sum of the star structures coming from Proposition \ref{prop:starmods-tensor-powers}.
\end{proof}

\begin{rem}
  \label{rem:generators-and-relations-and-stars}
  We now briefly discuss $*$-structures on algebras given by generators and relations.
  This means the following: we take a module $V$ (the generators) and a submodule $W \subseteq T(V)$ (the relations), and form the quotient algebra $A = T(V)/J_W$, where $J_W$ is the 2-sided ideal generated by $W$ as in Lemma \ref{lem:generating-star-ideals}.
  If $V$ is a $*$-module, then $T(V)$ is a $*$-algebra, and we would like to know when this $*$-structure descends to $A$.
  By Lemma \ref{lem:quotient-by-star-ideal}, it is sufficient for $J_W$ to be a $*$-ideal.
  Then by Lemma \ref{lem:generating-star-ideals}, it is sufficient that $W$ is a $*$-submodule of $T(V)$.
  According to Lemma \ref{lem:kernels-and-images-are-star-submodules}, this happens, for instance, when $W$ is the kernel or image of a morphism of $*$-modules.
  While these observations are more or less trivial in the abstract, they may be useful later on in specific circumstances when checking relations by hand is complicated.
\end{rem}

\begin{prop}
  \label{prop:tensor-alg-star-universal-prop}
  Let $V \in \hmods$, and endow $T(V)$ with the corresponding $*$-structure as described in Proposition \ref{prop:tensor-alg-star-struct}.
  Let $A$ be a $*$-algebra in $\hmod$ and let $f : V \to A$ be a $*$-module map.
  Then the morphism $\tilde{f} : T(V) \to A$ from Lemma \ref{lem:tensor-alg-universal-property} is in fact a morphism of $*$-algebras.
\end{prop}

\begin{proof}
  We know from Lemma \ref{lem:tensor-alg-universal-property} that $\tilde{f}$ is a morphism of module-algebras, so we only need to show that $\tilde{f}$ respects the $*$-structures on $T(V)$ and $A$.
  We claim that the following diagram commutes:
  
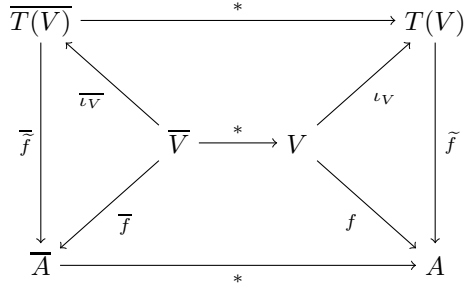
\begin{equation*}
    \label{eq:tensor-alg-star-struct-bigdiagram}
    \begin{tikzpicture}
      [description/.style={fill=white,inner sep=2pt}]
      \matrix (m) [matrix of math nodes, row sep=3em,
      column sep=3em, text height=2ex, text depth=0.25ex]
      { \bar{T(V)} & & & T(V)  \\
        & \bar{V} & V & \\
        \bar{A} & & & A \\ };
      \path[->,font=\scriptsize]
      (m-1-1) edge node[auto] {$*$} (m-1-4)
      (m-1-1) edge node[auto,swap] {$\bar{\widetilde{f}}$} (m-3-1)
      (m-2-2) edge node[auto] {$\bar{\iota_V}$} (m-1-1)
      (m-2-3) edge node[auto,swap] {$\iota_V$} (m-1-4)
      (m-1-4) edge node[auto] {$\widetilde{f}$} (m-3-4)
      (m-2-2) edge node[auto] {$*$} (m-2-3)
      (m-2-2) edge node[auto] {$\bar{f}$} (m-3-1)
      (m-2-3) edge node[auto,swap] {$f$} (m-3-4)
      (m-3-1) edge node[auto,swap] {$*$} (m-3-4);
    \end{tikzpicture}
  \end{equation*} 
  Indeed, the upper trapezoid commutes by Proposition \ref{prop:tensor-alg-star-struct}, the triangle on the right commutes by Lemma \ref{lem:tensor-alg-universal-property}, the lower trapezoid commutes since $f$ is a $*$-map, and the triangle on the left is just the complex conjugate of the triangle on the right.
  Thus the large rectangle commutes also, which means that $\tilde{f}$ is a $*$-map.
\end{proof}

\section{Inner products and adjoints}
\label{sec:inner-products}

We now turn to a discussion of inner products.
In the traditional formulation, an inner product on a complex vector space $V$ is a map $\rinprod : V \times V \to \bbC$ which is antilinear in the first variable, linear in the second variable, satisfies $\rinprod[v,w] = \bar{\rinprod[w,v]}$ for $v,w \in V$, and which is positive definite.
This definition, like that of a $*$-structure, does not lend itself easily to a module-theoretic approach because of the antilinearity.
Inner products allow one to introduce the notion of adjoint linear transformations; this too is an antilinear concept.
In this section we develop these notions in our framework.

Although so far we have discussed $*$-structures only on modules, not on arbitrary complex vector spaces, the discussion extends to complex vector spaces just by taking $H = \bbC$ with $*$-structure given by complex conjugation.
In particular, we can speak of $*$-vector spaces and morphisms of $*$-vector spaces.

\subsection{Inner products}
\label{sec:inner-product-def}

We now formulate the definition of an inner product in the language of complex conjugate vector spaces and modules:

\begin{dfn}
  \label{dfn:inner-product}
  For $V \in \cvect$ we define an \emph{inner product} on $V$ to be a linear map $\inprod : V^e = \bar{V} \otimes V \to \bbC$ such that
  \begin{enumerate}[(a)]
  \item With respect to the $*$-structures on $V^e$ and $\bbC$ constructed in Proposition \ref{prop:enveloping-modules} and Lemma \ref{lem:cself_conj}, respectively, $\inprod$ is a morphism of $*$-vector spaces.
  \item $\inprod[\bar{v},v] \ge 0$ for all $v \in V$, and $\inprod[\bar{v},v] = 0$ only if $v = 0$.
  \end{enumerate}
  If in addition $V$ is an $H$-module and if $\inprod$ is a module map, then we say that $\inprod$ is an inner product in $\hmod$; in this case we will call $V$ together with this inner product a \emph{Hermitian module}.
  If $V$ and $W$ are Hermitian modules with inner products $\inprodv$ and $\inprodw$, respectively, then we say that a module map $T : V \to W$ is an \emph{isometry} if $\inprodw[\bar{Tu},Tv] = \inprodv[\bar{u},v]$ for all $u,v \in V$.
  We say that $T$ is \emph{unitary} if $T$ is an invertible isometry.
\end{dfn}

\begin{rem}
  \label{rem:inner-prod-star-map}
  Untangling the definitions shows that requiring $\inprod$ to be a $*$-map means just that $\inprod[\bar{v},w] = \bar{\inprod[\bar{w},v]}$ for all $v,w \in V$.

  As we have done in previous sections, we could make the Hermitian modules into a category where the morphisms are the isometries.
  This is somewhat restrictive, as the isometry condition forces all morphisms to be injective, and for our purposes we don't need to consider this category separately, so we omit the definition.
\end{rem}

\begin{notn}
  \label{notation:inner-prods}
  As we did in Definition \ref{dfn:inner-product}, we will frequently abuse notation by writing $h(w,x)$ rather than $h(w \otimes x)$ when $h : W \otimes X \to Y$ is a linear map.
  This is nothing more than the canonical identification of bilinear maps $W \times X \to Y$ with linear ones $W \otimes X \to Y$, and we will make this identification without further comment from now on.
\end{notn}

We now show that Definition \ref{dfn:inner-product} is equivalent to the usual definition of an inner product, and we formulate the appropriate notion of $H$-invariance for the traditional incarnation $\rinprod : V \times V \to \bbC$ of the inner product in the case when $V$ is an $H$-module:

\begin{prop}
  \label{prop:constructing-inner-products}
  Let $V \in \cvect$ and let $\inprod : \bar{V} \otimes V \to \bbC$ be a linear map.
  Define a function $\rinprod : V \times V \to \bbC$ by $\rinprod[v,w] = \inprod[\bar{v},w]$ for $v,w \in V$.
  Then
  \begin{enumerate}[(a)]
  \item $\inprod$ is an inner product in the sense of Definition \ref{dfn:inner-product} if and only if $\rinprod$ satisfies the usual definition of an inner product given at the beginning of \S \ref{sec:inner-products}.
  \item If $V \in \hmod$, then $\inprod$ is a module map if and only if $\rinprod$ is $H$-invariant in the sense that $\rinprod[a \rhd v,w] = \rinprod[v, a^* \rhd w]$ for all $v, w \in V$ and $a \in H$.
  \end{enumerate}
\end{prop}

\begin{proof}
  \begin{enumerate}[(a)]
  \item Since $v \mapsto \bar{v}$ is antilinear, it is clear that $\rinprod[v,w]$ is antilinear in $v$, and linearity in $w$ is also clear.
    The formulations of positive-definiteness for $\inprod$ and $\rinprod$ are also clearly equivalent.
    We claim now that $\inprod$ is a $*$-map, i.e.~that the diagram 
    \[ 
    \begin{CD}
      \bar{\bar{V} \otimes V} @>{\bar{\inprod}}>> \bar{\bbC} \\
      @V{*}VV @VV{\gamma}V \\
      \bar{V} \otimes V @>{\inprod}>> \bbC
    \end{CD}
    \]
    commutes, if and only if $\rinprod$ satisfies conjugate-symmetry.
    For $v,w \in V$, chasing $\bar{\bar{v} \otimes w}$ right and then down in the diagram gives
    \[   \gamma(   c_\bbC( \inprod[\bar{v}, w])  ) =  \bar{\inprod[\bar{v}, w]} \overset{\mathrm{def}}{=} \bar{\rinprod[v,w]}.  \]
    On the other hand, chasing $\bar{\bar{v} \otimes w}$ down and then right gives
    \[ \inprod[ \bar{ \bar{v} \otimes w  }^* ] = \inprod[ \bar{w}, v ]  \overset{\mathrm{def}}{=} \rinprod[w,v],  \]
    so our claim is verified.
  \item First, observe that for $H$-modules $W$ and $X$, a linear map $h : W \otimes X \to \bbC$ (i.e.~a bilinear form) is a module map if and only if $h(a \rhd w, x) = h(w, S(a) \rhd x)$ for all $w \in W$ and $x \in X$.
    Then for $v,w \in V$ and $a \in H$ we have
    \[  \rinprod[a \rhd v, w] = \inprod[ \bar{a \rhd v}, w]  = \inprod[S(a)^* \rhd \bar{v}, w],  \]
    while on the other hand we have
    \[  \rinprod[v, a^* \rhd w] = \inprod[\bar{v}, a^* \rhd w].    \]
    Thus $\rinprod[a \rhd v, w] = \rinprod[v, a^* \rhd w]$ for all $a \in H$ if and only if $\inprod[S(a)^* \rhd \bar{v}, w] = \inprod[\bar{v}, a^* \rhd w]$ for all $a \in H$.  
    Replacing $a$ with $S(a)^*$ and using \eqref{eq:sstar}, the latter equation becomes $\inprod[a \rhd \bar{v}, w] = \inprod[\bar{v}, S(a) \rhd w]$.
    Then our first observation concludes the proof.  \qedhere
  \end{enumerate}
\end{proof}

\begin{prop}
  \label{prop:inner-product-dual-iso}
  Suppose that $V \in \hmod$ is a Hermitian module.
  Then $\mu_V : \bar{V} \to V^*$ given by $\mu_V(\bar{v}) = \inprod[\bar{v}, -]$ is an injective module map, and $\mu_V$ is an isomorphism if $V$ is finite-dimensional.
  If $W \in \hmod$ is another Hermitian module and the module map $T : V \to W$ is an isometry, then the following diagram commutes:
  \begin{equation}
    \label{eq:inner-prod-dual-iso-naturality}
    \begin{CD}
      \bar{V} @>{\bar{T}}>> \bar{W} \\
      @V{\mu_V}VV  @VV{\mu_W}V \\
      V^* @<<{T^{\tr}}< W^*   
    \end{CD}
  \end{equation}
\end{prop}

\begin{rem}
  \label{rem:inner-product-dual-iso-caveat}
  We caution the reader that $\mu_V$ is not intrinsic to $V$; it is only defined relative to a fixed inner product.
\end{rem}

\begin{proof}[Proof of Proposition \ref{prop:inner-product-dual-iso}]
  Since $\mu_V = \Psi_{V V^*}(v \mapsto \inprod[\bar{v},-])$, we need to show that $v \mapsto \inprod[\bar{v},-]$ is an antimodule map.
  For $a \in H$ and $v, w \in V$ we have
  \begin{align*}
    \inprod[\bar{a \rhd v}, w] & = \inprod[S(a)^* \rhd \bar{v}, w] \\
    & = \inprod[\bar{v}, S(S(a)^*) \rhd w] \\
    & = (S(a)^* \rhd \inprod[\bar{v}, -])(w),
  \end{align*}
  where for the second equality we used the fact that $\inprod$ is a module map, and for the third we used the definition \eqref{eq:dual-action} of the $H$-action on $V^*$.
  Hence $\mu_V$ is a module map by Proposition \ref{prop:antimodule-representable}.
  Injectivity follows immediately from the fact that $\inprod$ is positive-definite, and then $\mu_V$ must be an isomorphism when $V$ is finite-dimensional because the dimensions of $\bar{V}$ and $V^*$ coincide.

  Finally, for any $u,v \in V$, we have
  \begin{align*}
   [ T^{\tr} \mu_W \bar{T}(\bar{u})](v) & = [\mu_W (\bar{Tu})](Tv) \\
   & = \inprod[\bar{Tu},Tv] \\
   & = \inprod[\bar{u},v] \\
   & = [\mu_V(\bar{u})](v),
  \end{align*}
  so $T^{\tr}\mu_W \bar{T} = \mu_V$, and hence \eqref{eq:inner-prod-dual-iso-naturality} commutes.
\end{proof}

\begin{rem}
  \label{rem:inner-prod-left-vs-right}
  Many authors define an inner product to be linear in the first variable and antilinear in the second.
  While this makes no difference for vector spaces, it matters for modules: if we were to translate this alternate definition into our framework, an inner product would become a positive-definite $*$-module map $\inprod : V \otimes \bar{V} \to \bbC$.
  With this definition, Proposition \ref{prop:inner-product-dual-iso} would instead give an isomorphism $\bar{V} \simeq \us V$.
  All of the subsequent theory would carry over, \emph{mutatis mutandis}, but we omit it here for the sake of brevity.
\end{rem}

\subsection{Some remarks on positivity}
\label{sec:positivity}

We have seen in the proof of Proposition \ref{prop:constructing-inner-products} that the conjugate-symmetry condition $\rinprod[v,w] = \bar{\rinprod[w,v]}$ is encapsulated in the requirement that $\inprod$ respects the $*$-structures on $V^e = \bar{V} \otimes V$ and $\bbC$.
Thus we can express conjugate-symmetry entirely as a condition on maps in the category rather than as a condition involving elements.
This raises the question of whether the positivity criterion for an inner product can also be phrased in such a manner.
It appears that the answer is no, although this may be a failure of imagination rather than an insight into necessity.
One way to avoid this question entirely would be to include the notion of ``positive cone'' as part of the data when discussing inner products.
We will begin with an example and then give a general definition.

The fact that $\inprod$ is a $*$-map implies that any self-adjoint element of $V^e$ is mapped to $\bbR$ (the self-adjoint part of $\bbC$).
Positivity is a further restriction: it requires that the particular self-adjoint elements in $V^e$ of the form $\bar{v} \otimes v$ are mapped into the non-negative real numbers $\bbR^+$.
If we denote by $(V^e)^+$ the $\bbR^+$-span of the vectors $\bar{v} \otimes v$, then it is not difficult to convince oneself that $(V^e)^+$ spans the whole self-adjoint part of $V^e$ and that $(V^e)^+ \cap -(V^e)^+ = \set{0}$.
The non-negative reals $\bbR^+$ share the same properties in $\bbC$.
Then the inner product is a $*$-map which which takes $(V^e)^+$ into $\bbR^+$.
This motivates the following definition.

Consider the category $\hmodsp$ of pairs $(V,V^+)$ where $V$ is a $*$-module and $V^+$, the \emph{positive cone of $V$}, is a subset consisting of self-adjoint elements of $V$ which is closed under addition and multiplication by $\bbR^+$.
We also require that $V^+ \cap (-V^+) = \set{0}$ and that $V^+ - V^+ = V^{sa}$, the real subspace of self-adjoint elements of $V$.
A morphism $T : (V,V^+) \to (W,W^+)$ in $\hmodsp$ is a $*$-module map $T : V \to W$ such that $T(V^+) \subseteq W^+$.
Then there is a functor from $\hmod$ to $\hmodsp$ given on objects by $V \mapsto V^e = \bar{V} \otimes V$ and on morphisms by $T \mapsto \bar{T} \otimes T$, and we define the positive cone of $V^e$ to be the $\bbR^+$-span of vectors of the form $\bar{v} \otimes v$ for $v \in V$, as above.
The pair $(\bbC, \bbR^+)$ is an object of $\hmodsp$, and then an inner product on $V$ would be defined to be a morphism $\inprod$ in $\hmodsp$ from $V^e$ to $\bbC$, with positive cones as above.

The necessity to define the positive cone and to include that as part of the data manifests itself in other situations where there may be more than one relevant notion of positivity in a single vector space.
For example, in a $*$-algebra $A$ (in the usual sense, not in the sense presented here, although the discussion could be formulated in our setting as well) one can define the positive cone $A^+$ to be the $\bbR^+$-span of elements of the form $a^*a$.
Then we can say that a linear functional $f \in A^*$ is positive if $f(A^+) \subseteq \bbR^+$, and finally we can define another positive cone in $A$ by declaring $A^{++}$ to be the collection of self-adjoint elements of $A$ which are mapped to $\bbR^+$ by all positive functionals (and then check that this forms a cone).
Certainly $A^+ \subseteq A^{++}$, and there is equality if, for instance, $A$ is a $C^*$-algebra, but in general these notions will not coincide.

Although this would give an even more structural approach to the study of inner products, for our purposes we don't need to make things quite so abstract, so we confine this construction to this subsection.
We caution also that this approach has not been carefully considered by the author, and some tweaking of the conditions above may be necessary to develop the approach rigorously.

\subsection{Adjoints of linear operators}
\label{sec:adjoints}

Now we show how to use our notion of inner products to frame adjoints of linear maps in our language.
We begin with some motivational discussion on the usual formulation of adjoints.
For this we restrict to finite-dimensional modules.

\begin{rem}
  \label{rem:ordinary-adjoints}
  When $V$ and $W$ are finite-dimensional inner product spaces with (ordinary, sesquilinear) inner products $\rinprodv$ and $\rinprodw$, respectively, one defines the \emph{adjoint} of a linear transformation $T : V \to W$ to be the unique linear map $T^\dagger : W \to V$ satisfying 
  \begin{equation}
    \label{eq:adjoint-def-usual}
    \rinprodv[T^\dagger w, v] = \rinprodw[w,Tv] 
  \end{equation}
  for all $v \in V$ and $w \in W$.
  Then one shows that the following properties hold: 
  \begin{enumerate}[(a)]
  \item The map $\Hom(V,W) \to \Hom(W,V)$ given by $T \mapsto T^\dagger$ is antilinear.
  \item Taking the adjoint again gives $(T^\dagger)^\dagger = T$ in $\Hom(V,W)$.
  \item If $X$ is another finite-dimensional inner product space, then for any linear map $U : W \to X$, we have $(UT)^\dagger = T^\dagger U^\dagger$.
  \end{enumerate}
  Rather than repeating the construction of the adjoint operator from scratch, we use the correspondence from Proposition \ref{prop:constructing-inner-products} between sesquilinear inner products and inner products in our sense in order to transport the traditional notion of adjoint given above into our setting.
\end{rem}

\begin{notn}
  \label{conv:inner-prods}
  We make the following notational convention: for a Hermitian module (or just complex vector space) $V$ with inner product $\inprodv$ in the sense of Definition \ref{dfn:inner-product}, we will always denote by $\rinprodv$ the corresponding positive-definite, conjugate-symmetric, sesquilinear form given by
  \[  \rinprodv[v,w] = \inprodv[\bar{v},w]    \]
  for $v,w \in V$, as discussed in Proposition \ref{prop:constructing-inner-products}.
  With this convention in mind, translating the defining property \eqref{eq:adjoint-def-usual} of the adjoint transformation into our language gives
  \begin{equation}
    \label{eq:adjoint-def-new}
    \inprodv[\bar{T^\dagger w}, v] = \inprodw[\bar{w},Tv]
  \end{equation}
  for $v \in V$ and $w \in W$.
\end{notn}

In Proposition \ref{prop:adjoint-basic-properties} we explore some properties of the adjoint map $T^\dagger$.
Part (a) shows that the adjoint of a module map is a module map.
Part (b), while somewhat technical in its statement, is really just formalizing the notion that the adjoint of a linear map is its conjugate transpose.

\begin{prop}
  \label{prop:adjoint-basic-properties}
  Let $V, W \in \hmodf$ be finite-dimensional Hermitian modules, let $T : V \to W$ be a linear map, and let $T^\dagger : W \to V$ be the adjoint of $T$.  
  Then:
  \begin{enumerate}[(a)]
  \item $T$ is a module map if and only if $T^\dagger$ is a module map.
  \item $T^\dagger$ coincides with the composition
   \[  
   \begin{CD}
     W @>{\sigma_W^{-1}}>> \bar{\bar{W}} @>{\bar{\mu_W}}>> \bar{W^*} @>{\bar{T^{\tr}}}>> \bar{V^*} @>{\bar{\mu_V^{-1}}}>> \bar{\bar{V}} @>{\sigma_V}>> V
   \end{CD}
   \] 
  \end{enumerate}
\end{prop}

\begin{proof}
  \begin{enumerate}[(a)]
  \item Since $(T^\dagger)^\dagger = T$, we only need to do one direction of the proof.
    Assuming that $T$ is a module map, we need to show that $T^\dagger(a \rhd w) = a \rhd (T^\dagger w)$ for $w \in W$ and $a \in H$.
    For any $v \in V$, we have
    \begin{align*}
      \rinprodv[T^\dagger(a \rhd w), v] & = \rinprodw[a \rhd w, Tv] \\
      & = \rinprodw[w, a^* \rhd (Tv) ] \\
      & = \rinprodw[w, T(a^* \rhd v)] \\
      & = \rinprodv[Tw,a^* \rhd v] \\
      & = \rinprodv[a \rhd (Tw),v],
    \end{align*}
    using \eqref{eq:adjoint-def-usual}, Proposition \ref{prop:constructing-inner-products} (b), and the fact that $T$ is a module map.
    Since this holds for all $v \in V$, we conclude that $T^\dagger(a \rhd w) = a \rhd (T^\dagger w)$, so $T^\dagger$ is a module map.
  \item
    Recall that $c_{W^*}$ is the canonical antimodule map from $W^*$ to its conjugate.
    For $w \in W$, we have
    \[  \bar{\mu_W} (\sigma_W^{-1}(w)) = c_{W^*}( \inprodw[\bar{w},-] ).   \]
    Applying $\bar{T^{\tr}}$ to this gives
    \begin{align*}
      \bar{T^{\tr}} (c_{W^*}(\inprodw[\bar{w},-])) & = c_{V^*}\left(\inprodw[\bar{w},T(-)]\right) \\
      & = c_{V^*} \left(\inprodv[\bar{T^\dagger w},-]\right).
    \end{align*}
    Finally, applying $\sigma_V \bar{\mu_V^{-1}}$ now gives
    \begin{align*}
      \sigma_V \bar{\mu_V^{-1}} \left(c_{V^*} \left(\inprodv[\bar{T^\dagger w},-]\right) \right) & = \sigma_V( \bar{\bar{T^\dagger w}}) \\
      & = T^\dagger w,
    \end{align*}
    which proves our claim.  \qedhere
  \end{enumerate}
\end{proof}

We now turn to the module-theoretic properties of the map $T \mapsto T^\dagger$.
Recall from Definition \ref{dfn:hom-actions} that we have two actions of $H$ on each of the linear spaces $\Hom(V,W)$ and $\Hom(W,V)$ which we called \emph{left} and \emph{right}.
We know that taking the adjoint is an antilinear map $\Hom(V,W) \to \Hom(W,V)$, but \emph{a priori} it is not clear whether this should be an antimodule map for the left or right actions on the two $\Hom$-spaces.

It turns out that we need to use the left action for both of them.
The point here is that since we are dealing with finite-dimensional modules, we have isomorphisms $\bar{V} \simeq V^*$ and $\homl(V,W) \simeq W \otimes V^*$ from Propositions \ref{prop:inner-product-dual-iso} and \ref{prop:hom-decomps}, respectively.
Combining these gives an isomorphism $\homl(V,W) \simeq W \otimes \bar{V}$, and similarly we have $\homl(W,V) \simeq V \otimes \bar{W}$.
As in Remark \ref{rem:inner-prod-left-vs-right}, we can see that this is a result of our convention that inner products are maps defined on $\bar{V} \otimes V$; if we had chosen the opposite convention then we would have $\homr$ instead of $\homl$ in the following:

\begin{prop}
  \label{prop:adjoint-module-properties}
  Let $V,W \in \hmodf$ be finite-dimensional Hermitian modules.
  \begin{enumerate}[(a)]
  \item The map $\homl(V,W) \to \homl(W,V)$ given by $T \mapsto T^\dagger$ is an anti-isomorphism of modules.
  \item The map $* : \bar{\homl(V,W)} \to \homl(W,V)$ associated to $T \mapsto T^\dagger$ by Proposition \ref{prop:antimodule-representable} is an isomorphism of modules.
  \end{enumerate}
\end{prop}

\begin{proof}
  \begin{enumerate}[(a)]
  \item We need to show that $(a \rhd T)^\dagger = S(a)^* \rhd T^\dagger$ for $T \in \homl(V,W)$ and $a \in H$, where the action of $H$ on $\homl(V,W)$ is given by \eqref{eq:hom-action}.
    We will do this by showing that
    \[ \rinprodv[(a \rhd T)^\dagger w, v] = \rinprodv[(S(a)^* \rhd T^\dagger) w,v] \]
    for all $v \in V$, $w \in W$, and $a \in H$.
    Then the result will follow from nondegeneracy of the sesquilinear form.
    For the sake of readability, we omit the $\rhd$ symbols for the $H$-actions on $V$ and $W$ in the following.
    Beginning with the right-hand side, we have
    \begin{align*}
      \rinprodv[(S(a)^* \rhd T^\dagger) w,v] & = \rinprodv[S(a_{(2)})^* T^\dagger S(S(a_{(1)})^*)w, v] \\
      & = \rinprodv[T^\dagger a_{(1)}^* w, S(a_{(2)})v] \\
      & = \rinprodw[w,a_{(1)}TS(a_{(2)})v] \\
      & = \rinprodw[w, (a\rhd T)v] \\
      & = \rinprodv[(a \rhd T)^\dagger w, v],
    \end{align*}
    where we used the property \eqref{eq:adjoint-def-usual} of the adjoint, as well as the invariance of $\rinprodv$ and $\rinprodw$ under the $H$-action described in part (b) of Proposition \ref{prop:constructing-inner-products}.

  \item Follows immediately from (a) together with Proposition \ref{prop:antimodule-representable}.  \qedhere
  \end{enumerate}
\end{proof}

\begin{rem}
  \label{rem:adjoint-not-star}
  Although the notation is suggestive, the map $* : \bar{\homl(V,W)} \to \homl(W,V)$ is not a $*$-structure in the sense of Definition \ref{dfn:star-module} because it is not of the form $* : \bar{Y} \to Y$ for a module $Y$ except when $V = W$.
  We address this case now.
  Since we have not yet shown that $\Endl(V)$ is an $H$-module algebra, we take care of this detail first.
\end{rem}

\begin{prop}
  \label{prop:endV-module-algebra}
  For $V \in \hmod$, the algebra $\Endl(V)$ is an algebra in $\hmod$.
\end{prop}

\begin{proof}
  We need to show that the conditions \eqref{eq:module-alg-conditions} hold.
  For $x,y \in \Endl(V)$ and $a \in H$ we have
  \begin{align*}
    (a_{(1)} \rhd x)(a_{(2)} \rhd y) & = a_{(1)} x S(a_{(2)}) a_{(3)} y S(a_{(4)}) \\
    & = a_{(1)}x \counit (a_{(2)}) y S(a_{(3)}) \\
    & = a_{(1)} xy S(\counit(a_{(2)}) a_{(3)}) \\
    & = a_{(1)} xy S(a_{(2)}) \\
    & = a \rhd (xy),
  \end{align*}
  so multiplication in $\Endl(V)$ is a module map.
  For the unit map, we write $1 = \id_V$; we have
  \[ a \rhd 1 = a_{(1)}  1 S(a_{(2)}) = \counit(a)1, \]
  so the unit is also a module map, which concludes the proof.
\end{proof}

\begin{prop}
  \label{prop:endV-star-structure}
  Let $V \in \hmodf$ be a finite-dimensional Hermitian module.
  Then the map $* : \bar{\Endl(V)} \to \Endl(V)$ from Proposition \ref{prop:adjoint-module-properties} makes $\Endl(V)$ into a $*$-algebra in $\hmod$.
\end{prop}

\begin{proof}
  We know already from Proposition \ref{prop:endV-module-algebra} that $\Endl(V)$ is an $H$-module algebra.
  We know also from Proposition \ref{prop:adjoint-module-properties} that $T \mapsto T^\dagger$ is an anti-isomorphism of modules, and according to the properties (b) and (c) listed in Remark \ref{rem:ordinary-adjoints}, we see that this map is actually an involutive antimodule-algebra morphism.
  Thus $*$ is a $*$-structure by Proposition \ref{prop:constructing-star-algebras}.
\end{proof}

\begin{rem}
  \label{rem:adjoint-connection-enveloping-module}
  As in the discussion preceding Proposition \ref{prop:adjoint-module-properties}, note that $\Endl(V) \simeq V \otimes \bar{V}$, which we denoted by $\ue V$ in \S \ref{sec:building-up-hmods}.
  We note that the $*$-structure on $\Endl(V)$ from Proposition \ref{prop:endV-star-structure} corresponds through this isomorphism with the $*$-structure on $\ue V$ defined in Proposition \ref{prop:enveloping-modules}.
\end{rem}

\subsection{Inner products, $*$-structures, and bilinear forms}
\label{sec:inprods-stars-blfs}

Let $V \in \hmodf$.
In this subsection we explore the relationship between inner products, $*$-structures, and bilinear forms on $V$.

By definition, a $*$-structure on $V$ is an isomorphism $* : \bar{V} \to V$ with a certain extra property (namely involutivity).
If $V$ is also a Hermitian module, then we can define a bilinear form $h : V \otimes V \to \bbC$ by $h(\bar{v}^*,w) = \inprod[\bar{v},w]$.
In other words, we define the bilinear form so that the following diagram commutes:
\begin{equation}
  \label{eq:blf-from-star-and-inprod}
  \begin{tikzpicture}
    [baseline=(current bounding box.center)
    description/.style={fill=white,inner sep=2pt}]
    \matrix (m) [matrix of math nodes, row sep=3em,
    column sep=3em, text height=2ex, text depth=0.25ex]
    { \bar{V}\otimes V & V \otimes V \\
      & \bbC \\ };
    \path[->,font=\scriptsize]
    (m-1-1) edge node[auto] {$ * \otimes \id$} (m-1-2)
    (m-1-1) edge node[auto,swap] {$\inprod$} (m-2-2)
    (m-1-2) edge node[auto] {$h$} (m-2-2);
  \end{tikzpicture}
\end{equation} 
Since $*$ is an isomorphism and since the inner product is nondegenerate, the bilinear form $h$ will be nondegenerate as well.
Note that $h$ is a morphism in $\hmod$.

\begin{rem}
  \label{rem:two-out-of-three}
  A $*$-structure on $V$ gives an isomorphism $\bar{V} \simeq V$.
  An inner product on $V$ gives an isomorphism $\bar{V} \simeq V^*$ via $\bar{v} \mapsto \inprod[\bar{v},-]$; see Proposition \ref{prop:inner-product-dual-iso}.
  Similarly, a nondegenerate bilinear form $h : V \otimes V \to \bbC$ gives an isomorphism $V \simeq V^*$ via $v \mapsto h(v,-)$.
  
  We saw above that having a $*$-structure and an inner product allowed us to obtain a bilinear form.
  This corresponds to composing the corresponding isomorphisms $V \simeq \bar{V}$ and $\bar{V} \simeq V^*$ to obtain the isomorphism $V \simeq V^*$ associated to the bilinear form.
  We would now like to ask if there is a ``two-out-of-three'' type result, i.e.~whether having any two of a bilinear form, an inner product, and a $*$-structure, allows us to obtain the third one.

  In the discussion preceding this remark, we got the bilinear form for free.
  But we will see that in fact this bilinear form necessarily satisfies some conditions, so it is not the case that, for example, an arbitrary bilinear form plus an inner product will give a $*$-structure.
  Similarly, we cannot expect an arbitrary bilinear form plus a $*$-structure to give an inner product.
  Thus a completely general two-out-of-three result is not possible.
  But if we restrict the bilinear forms we consider, we do obtain some results.
\end{rem}

\begin{prop}
  \label{prop:inprods-stars-blfs}
  Let $V \in \hmodf$ and let $* : \bar{V} \to V$, $\inprod : \bar{V} \otimes V \to \bbC$, and $h : V \otimes V \to \bbC$ be module maps such that the diagram \eqref{eq:blf-from-star-and-inprod} commutes (we do not assume $*$ to be a $*$-structure nor $\inprod$ to be an inner product).
  Then
  \begin{enumerate}[(a)]
  \item If $\inprod$ is an inner product on $V$, then $*$ is a $*$-structure on $V$ if and only if $h(\bar{w}^*, \bar{v}^*) = \bar{h(v,w)}$ for all $v,w \in V$.
  \item If $*$ is a $*$-structure on $V$, then $\inprod$ is an inner product on $V$ if and only if $h$ satisfies the conditions $h(\bar{w}^*, \bar{v}^*) = \bar{h(v,w)}$ for $v,w \in V$, $h(\bar{v}^*, v) \ge 0$ for $v \in V$, and $h(\bar{v}^*, v) = 0$ only if $v = 0$.
  \end{enumerate}
\end{prop}

\begin{proof}
  Consider the following diagram:
  \begin{equation}
    \label{eq:two-out-of-three-diagram}
    \begin{tikzpicture}
      [baseline=(current bounding box.center)
      description/.style={fill=white,inner sep=2pt}]
      \matrix (m) [matrix of math nodes, row sep=4em,
      column sep=3em, text height=2ex, text depth=0.25ex]
      { \bar{\bar{V} \otimes V} & & \bar{V} \otimes \bar{\bar{V}} & &  \bar{V} \otimes V \\
        & \bar{V \otimes V} & \bar{V} \otimes \bar{V} & V \otimes V & \\ 
        \bar{\bbC} & & & & \bbC \\   };
      \path[->,font=\scriptsize]
      (m-1-1) edge node[auto] {$\rho_{\bar{V} V}$} (m-1-3)
      (m-1-1) edge node[auto] {$\bar{* \otimes \id}$} (m-2-2)
      (m-1-1) edge node[auto,swap] {$\bar{\inprod}$} (m-3-1)
      (m-1-3) edge node[auto] {$\id \otimes \sigma_V$} (m-1-5)
      (m-1-3) edge node[auto,swap] {$\id \otimes \bar{*}$} (m-2-3)
      (m-1-5) edge node[auto] {$* \otimes \id$} (m-2-4)
      (m-1-5) edge node[auto] {$\inprod$} (m-3-5)
      (m-2-2) edge node[auto,swap] {$\rho_{V V}$} (m-2-3)
      (m-2-2) edge node[auto] {$\bar{h}$} (m-3-1)
      (m-2-3) edge node[auto] {$\id \otimes *$} (m-1-5)
      (m-2-3) edge node[auto,swap] {$* \otimes *$} (m-2-4)
      (m-2-4) edge node[auto,swap] {$h$} (m-3-5)
      (m-3-1) edge node[auto,swap] {$\gamma$} (m-3-5);
    \end{tikzpicture}  
  \end{equation}

  Certain parts of this diagram commute automatically.
  The triangle on the right-hand side commutes since it is exactly \eqref{eq:blf-from-star-and-inprod}, and the triangle on the left-hand side is the complex conjugate of that on the right, so it commutes as well.
  The trapezoid on the upper left commutes by Proposition \ref{prop:conj_tensor_prod}.
  The small interior triangle commutes trivially.

  This leaves the triangle on the upper right, the pentagon, and the large rectangle.
  The triangle on the upper right commutes if and only if $*$ is a $*$-structure.
  The pentagon commutes if and only if $h(\bar{w}^*, \bar{v}^*) = \bar{h(v,w)}$ for all $v,w \in V$.
  Noting that the composition across the top of \eqref{eq:two-out-of-three-diagram} is exactly the $*$-structure on $\bar{V} \otimes V$ described in Proposition \ref{prop:enveloping-modules}, we see that the large rectangle commutes if and only if $h$ is a $*$-map.

  If $\inprod$ is an inner product, it is a $*$-map, and hence the large rectangle commutes.
  Then the upper right triangle commutes if and only if the pentagon commutes, i.e.~$*$ is a $*$-structure if and only if $h(\bar{w}^*, \bar{v}^*) = \bar{h(v,w)}$ for all $v,w \in V$.
  This proves (a).

  On the other hand, if $*$ is a $*$-structure, then the upper right triangle commutes.
  Thus the large rectangle commutes if and only if the pentagon commutes, i.e.~$\inprod$ is a $*$-map if and only if $h(\bar{w}^*, \bar{v}^*) = \bar{h(v,w)}$ for all $v,w \in V$.
  From \eqref{eq:blf-from-star-and-inprod} we have $h(\bar{v}^*,w) = \inprod[\bar{v},w]$, so we see that $\inprod$ satisfies the appropriate positivity requirement if and only if $h(\bar{v}^*, v) \ge 0$ for $v \in V$ and $h(\bar{v}^*, v) = 0$ only if $v = 0$.
  This proves (b).
\end{proof}

\begin{rem}
  \label{rem:two-out-of-three-proof}
  We now explain the meaning of the condition $h(\bar{w}^*, \bar{v}^*) = \bar{h(v,w)}$.
  We said in the proof of Proposition \ref{prop:inprods-stars-blfs} that this condition is equivalent to commutativity of the pentagon in the diagram \eqref{eq:two-out-of-three-diagram}.
  If $*$ is a $*$-structure on $V$, then the composition $(* \otimes *) \circ \rho_{VV}$ (i.e.~the top of the pentagon) is the $*$-structure on $V \otimes V$ described in \ref{prop:starmods-tensor-powers} (for $n=2$).
  Then the statement that the pentagon commutes means exactly that $h : V \otimes V \to \bbC$ is a $*$-map.
  We could also formulate the positivity criterion $h(\bar{v}^*,v) \ge 0$ in terms of positive cones, as discussed in \S \ref{sec:positivity}.
\end{rem}

\section{$*$-structures and braidings}
\label{sec:stars-and-braidings}

In this section we examine the interaction between \(\ast\)-structures, R-matrices, and braidings.
We show that the complex conjugate of a braiding on the module category of a Hopf \(\ast\)-algebra is again a braiding.
When the braiding comes from an R-matrix, we show how conditions on the R-matrix translate into relations between the braiding and its conjugate, and relations between the braiding and the \(\ast\)-structure on modules.
In particular, when the R-matrix is \emph{real} (see Definition \ref{dfn:rmatrix-real-inverse-real}), then the braiding of \(V\) with itself is a module map for any \(\ast\)-module \(V\).

\subsection{Braidings and their complex conjugates}
\label{sec:conj-of-braidings}

In this subsection we define a braiding and show that the complex conjugate of a braiding on a category of $H$-modules gives another braiding.
Although the notion of braiding makes sense for any monoidal category, we define it here only for $\hmod$ (or for a sub-monoidal category of $\hmod$).
This simplifies our presentation because we already understand the monoidal structure well: it is the tensor product of vector spaces over $\bbC$.
Thus we can and do suppress the associativity isomorphisms $U \otimes (V \otimes W) \simeq (U \otimes V) \otimes W$.

\begin{dfn}
  \label{dfn:braiding}
  Let $\subcat$ be a subcategory of $\hmod$ containing $\bbC$ and such that $V \otimes W \in \subcat$ for all $V,W \in \subcat$ (a \emph{sub-monoidal category}).
  A \emph{braiding} on $\cC$ is a collection of isomorphisms $\psi_{VW} : V \otimes W \to W \otimes V$ (in $\subcat$) for all $V,W \in \subcat$ satisfying the conditions
  \begin{equation}
    \label{eq:braid-axioms}
    \psi_{U,V \otimes W} = (\id_V \otimes \psi_{UW})(\psi_{UV} \otimes \id_W), \quad \psi_{U \otimes V, W} = (\psi_{U,W} \otimes \id_V)(\id_U \otimes \psi_{VW}),
  \end{equation}
  and such that for any modules $U,V,W,X$ and any morphisms $f : U \to W$ and $g : V \to X$ in $\subcat$, the following diagram commutes:
  \begin{equation}
    \label{eq:braiding-naturality}
    \begin{CD}
      U \otimes V @>{\psi_{UV}}>> V \otimes U \\
      @V{f \otimes g}VV @VV{g \otimes f}V \\
      W \otimes X @>>{\psi_{WX}}> X \otimes W
    \end{CD}
  \end{equation}
  We refer to the property \eqref{eq:braiding-naturality} as \emph{naturality} of the braiding.
  We refer to the relations \eqref{eq:braid-axioms} as the \emph{hexagon axioms}; the reason for this is that the relevant diagrams would each form a hexagon if we did not suppress the associators.
  \end{dfn}

\begin{rem}
  \label{rem:braiding-technical-def}
  A fancier way to define a braiding is as follows.
  One can form the Cartesian product category $\subcat \times \subcat$.
  Then the tensor product in $\subcat$ defines a functor $\bigotimes : \subcat \times \subcat \to \subcat$, which takes a pair of objects $(U,V)$ to their tensor product $U \otimes V$, and a pair of morphisms $(f,g)$ to the tensor product morphism $f \otimes g$.
  The fact that the tensor product is associative means that there is a natural isomorphism between the functors $\bigotimes \circ (\bigotimes \times \id)$ and $\bigotimes \circ (\id \times \bigotimes)$ from $\subcat \times \subcat \times \subcat$ to $\subcat$.
  Similarly, there is a functor $\bigotimes^{op}$ which takes the pairs $(U,V)$ and $(f,g)$ to $V \otimes U$ and $g \otimes f$, respectively.
  Then a braiding can be defined as a natural isomorphism between the functors $\bigotimes$ and $\bigotimes^{op}$.
  
  It follows from the axioms that the braidings $\psi_{\bbC V}$ and $\psi_{V \bbC}$ are compatible with the canonical identifications of $\bbC \otimes V$ and $V \otimes \bbC$ with $V$.
  We refer to \S 2 of \cite{JoyStr93} for further details on braidings in general.
\end{rem}

Our task now is to show that complex conjugation of a braiding is again a braiding.
While somewhat tedious, this is essentially straightforward.
For the rest of $\S \ref{sec:conj-of-braidings}$, assume that $\subcat$ is subcategory of $\hmod$ which is equipped with a braiding $\psi = (\psi_{VW})_{V,W \in \subcat}$ as in Definition \ref{dfn:braiding}.
Suppose furthermore that $\subcat$ is closed under complex conjugation, i.e.~that $\bar{V} \in \subcat$ for all $V \in \subcat$ and $\bar{f}$ is a morphism in $\subcat$ for all morphisms $f$ in $\subcat$.
Finally, assume that the isomorphisms $\rho_{VW} : \bar{V \otimes W} \to \bar{W} \otimes \bar{V}$ and $\sigma_V : \bar{\bar{V}} \to V$ from Proposition \ref{prop:conj_tensor_prod} and Lemma \ref{lem:conj_involutory}, respectively, along with their inverses, are in $\subcat$ as well.

We now define a new family of isomorphisms obtained from the original braiding $\psi$ by conjugation.
As this is somewhat technical, we make the definition in two stages.
First, for each $X,Y \in \subcat$, we define an isomorphism $\xi_{XY} : \bar{X} \otimes \bar{Y} \to \bar{Y} \otimes \bar{X}$ via the composition
\begin{equation}
  \label{eq:conj-braiding-step1}
  \begin{CD}
      \xi_{XY} : \bar{X} \otimes \bar{Y} @>{\rho^{-1}_{YX}}>> \bar{Y \otimes X}  @>{\bar{\psi_{XY}}}>>   \bar{X \otimes Y} @>{\rho_{XY}}>>  \bar{Y} \otimes \bar{X}.
  \end{CD}
\end{equation}
Note that our assumptions ensure that $\xi_{XY}$ is a morphism in $\subcat$.

We have now provided an isomorphism $\xi_{XY}$ between the modules $\bar{X} \otimes \bar{Y}$ and $\bar{Y} \otimes \bar{X}$ for any $X$ and $Y$.
In order to define the conjugated braiding, we use the fact that every module is naturally isomorphic to the conjugate of some module (namely to the conjugate of its conjugate):

\begin{dfn}
  \label{dfn:conjugated-braiding}
  For any $X,Y \in \subcat$ we define an isomorphism $\bar{\psi}_{XY}$ (not to be confused with $\bar{\psi_{XY}}$) from $X \otimes Y$ to $Y \otimes X$ via the composition
  \begin{equation}
    \label{eq:conjugated-braiding}
    \begin{CD}
      \bar{\psi}_{XY} : X \otimes Y @>{\sigma_X^{-1} \otimes \sigma_Y^{-1}}>> \bar{\bar{X}} \otimes \bar{\bar{Y}} @>{\xi_{\bar{X} \, \bar{Y}}}>> \bar{\bar{Y}} \otimes \bar{\bar{X}} @>{\sigma_Y \otimes \sigma_X}>> Y \otimes X.
    \end{CD}
  \end{equation}
  We refer to the family $\bar{\psi} = (\bar{\psi}_{XY})_{X,Y \in \subcat}$ as the \emph{conjugate braiding} to $\psi$.
\end{dfn}
As the name suggests, the family $\bar{\psi}$ is also a braiding on $\subcat$.
Verifying the details is a little tedious.
We begin by proving the relevant properties of the maps $\xi_{XY}$ before using them to establish the properties of $\bar{\psi}$.
First we show that the $\xi_{XY}$'s are natural with respect to module maps.

\begin{lem}
  \label{lem:xis-are-natural}
  Let $U,V,W,X \in \subcat$ and let $f : U \to W$ and $g : V \to X$ be morphisms in $\subcat$.
  Then the following diagram commutes:
  \begin{equation}
    \label{eq:xis-are-natural}
    \begin{CD}
      \bar{U} \otimes \bar{V} @>{\xi_{UV}}>> \bar{V} \otimes \bar{U} \\
      @V{\bar{f} \otimes \bar{g}}VV @VV{\bar{g} \otimes \bar{f}}V \\
      \bar{W} \otimes \bar{X} @>>{\xi_{WX}}> \bar{X} \otimes \bar{W}
    \end{CD}
  \end{equation}
\end{lem}

\begin{proof}
  This follows from \eqref{eq:braiding-naturality} together with \eqref{eq:conj_tensor_prod} (used twice).
\end{proof}

The harder part of showing that $\bar{\psi}$ is a braiding is checking that the hexagon axioms \eqref{eq:braid-axioms} hold.
We now prove a technical lemma, which will assist us in checking these relations.

\begin{lem}
  \label{lem:xi-hexagons}
  Let $U,V,W \in \subcat$.
  The following diagram commutes:
  \begin{equation}
    \label{eq:xi-hexagons}
    \begin{tikzpicture}
      [baseline=(current bounding box.center)
      description/.style={fill=white,inner sep=2pt}]
      \matrix (m) [matrix of math nodes, row sep=3em,
      column sep=3em, text height=2ex, text depth=0.25ex]
      { \bar{U} \otimes \bar{V \otimes W} & & \bar{V \otimes W} \otimes \bar{U} \\
        \bar{U} \otimes \bar{W} \otimes \bar{V} & \bar{W} \otimes \bar{U} \otimes \bar{V} & \bar{W} \otimes \bar{V} \otimes \bar{U} \\  };
      \path[->,font=\scriptsize]
      (m-1-1) edge node[auto] {$\xi_{U,V \otimes W}$} (m-1-3)
      (m-1-1) edge node[auto,swap] {$\id_{\bar{U}} \otimes \rho_{VW}$} (m-2-1)
      (m-1-3) edge node[auto] {$\rho_{VW} \otimes \id_{\bar{U}}$} (m-2-3)
      (m-2-1) edge node[auto,swap] {$\xi_{UW} \otimes \id_{\bar{V}}$} (m-2-2)
      (m-2-2) edge node[auto,swap] {$\id_{\bar{W}} \otimes \xi_{UV}$} (m-2-3);
    \end{tikzpicture}
  \end{equation}
\end{lem}

\begin{proof}
  Commutativity of \eqref{eq:xi-hexagons} follows from commutativity of the diagram \eqref{eq:xi-hexagons-proof}.
  Indeed, expanding all of the $\xi$'s in \eqref{eq:xi-hexagons} according to the definition gives the outer edges of \eqref{eq:xi-hexagons-proof}, so if \eqref{eq:xi-hexagons-proof} commutes then \eqref{eq:xi-hexagons} will as well.
  \begin{equation}
    \label{eq:xi-hexagons-proof}
    \begin{tikzpicture}
      [baseline=(current bounding box.center)
      description/.style={fill=white,inner sep=2pt}]
      \matrix (m) [matrix of math nodes, row sep=3em,
      column sep=2em, text height=2ex, text depth=0.25ex]
      { \bar{U} \otimes \bar{V \otimes W} & \bar{V \otimes W \otimes U} & & \bar{U \otimes V \otimes W} & \bar{V \otimes W} \otimes \bar{U} & \bar{W} \otimes \bar{V} \otimes \bar{U} \\
        & & \bar{V \otimes U \otimes W} & & & \\
        \bar{U} \otimes \bar{W} \otimes \bar{V} & \bar{W \otimes U} \otimes \bar{V} & \bar{U \otimes W} \otimes \bar{V} & \bar{W} \otimes \bar{U} \otimes \bar{V} & \bar{W} \otimes \bar{V \otimes U} & \bar{W} \otimes \bar{U \otimes V} \\
 };
      \path[->,font=\scriptsize]
      (m-1-1) edge node[auto] {$\rho^{-1}$} (m-1-2)
      (m-1-1) edge node[auto] {$\id \otimes \rho$} (m-3-1)
      (m-1-2) edge node[auto] {$\bar{\psi}$} (m-1-4)
      (m-1-2) edge node[auto] {$\bar{\id \otimes \psi}$} (m-2-3)
      (m-1-2) edge node[auto] {$\rho$} (m-3-2)
      (m-1-4) edge node[auto] {$\rho$} (m-1-5)
      (m-1-4) edge node[auto] {$\rho$} (m-3-6)
      (m-1-5) edge node[auto] {$\rho \otimes \id$} (m-1-6)
      (m-2-3) edge node[auto] {$\bar{\psi \otimes \id}$} (m-1-4)
      (m-2-3) edge node[auto] {$\rho$} (m-3-3)
      (m-2-3) edge node[auto] {$\rho$} (m-3-5)
      (m-3-1) edge node[auto,swap] {$\rho^{-1} \otimes \id$} (m-3-2)
      (m-3-2) edge node[auto,swap] {$\bar{\psi} \otimes \id$} (m-3-3)
      (m-3-3) edge node[auto,swap] {$\rho \otimes \id$} (m-3-4)
      (m-3-4) edge node[auto,swap] {$\id \otimes \rho^{-1}$} (m-3-5)
      (m-3-5) edge node[auto,swap] {$\id \otimes \bar{\psi}$} (m-3-6)
      (m-3-6) edge node[auto] {$\id \otimes \rho$} (m-1-6);
    \end{tikzpicture}    
  \end{equation}
  We have omitted the subscripts labeling the objects on the $\rho$'s and $\psi$'s for legibility.
  This should cause no confusion because there is only one choice of subscripts in each case that matches the domain and codomain of each arrow.

  It is left to prove that the diagram above commutes.
  We begin at the left-hand side and move to the right.
  The rectangle commutes by \eqref{eq:conj_tensor_prod_assoc}.
  The trapezoid commutes by \eqref{eq:conj_tensor_prod}.
  The upper triangle commutes since $\psi$ is a braiding.
  The lower triangle commutes by \eqref{eq:conj_tensor_prod_assoc}.
  The quadrilateral commutes by \eqref{eq:conj_tensor_prod}.
  Finally, the triangle on the right-hand side commutes by \eqref{eq:conj_tensor_prod_assoc}.
  This completes the proof of the lemma.
\end{proof}

We are now ready to prove our first main result in this section:
\begin{prop}
  \label{prop:conjugate-braiding-is-a-braiding}
  The conjugate braiding $\bar{\psi}$ is a braiding on $\subcat$.
\end{prop}

\begin{proof}
  First we check the naturality property \eqref{eq:braiding-naturality}.
  Let $U,V,W,X \in \subcat$ and let $f : U \to W$ and $g : V \to X$ be morphisms in $\subcat$.
  Then we claim that the following diagram commutes (where we have removed most of the subscripts on the maps for legibility):
  \begin{equation}
    \label{eq:conjugate-braiding-naturality}
    \begin{CD}
      U \otimes V @>{\sigma^{-1} \otimes \sigma^{-1}}>> \bar{\bar{U}} \otimes \bar{\bar{V}} @>{\xi_{\bar{U} \, \bar{V}}}>> \bar{\bar{V}} \otimes \bar{\bar{U}} @>{\sigma \otimes \sigma}>> V \otimes U \\
      @V{f \otimes g}VV @V{\bar{\bar{f}} \otimes \bar{\bar{g}}}VV @VV{\bar{\bar{g}} \otimes \bar{\bar{f}}}V @VV{g \otimes f}V \\
      W \otimes X @>>{\sigma^{-1} \otimes \sigma^{-1}}> \bar{\bar{W}} \otimes \bar{\bar{X}} @>>{\xi_{\bar{W} \, \bar{X}}}> \bar{\bar{X}} \otimes \bar{\bar{W}} @>>{\sigma \otimes \sigma}> X \otimes W
    \end{CD}
  \end{equation}
  Each of the two outer squares is a tensor product of two diagrams of the form \eqref{eq:conj_involutory}, so they commute by Lemma \ref{lem:conj_involutory}.
  The central square commutes by Lemma \ref{lem:xis-are-natural} applied to the modules $\bar{U},\bar{V},\bar{W},\bar{X}$ and the morphisms $\bar{f}$ and $\bar{g}$.
  Noting that the compositions across the top and bottom of \eqref{eq:conjugate-braiding-naturality} are exactly $\bar{\psi}_{UV}$ and $\bar{\psi}_{WX}$, respectively, this completes the proof of naturality.

  For the hexagon axioms, we will verify only the equality
  \begin{equation}
    \label{eq:conjugated-braiding-hexagon}
    \bar{\psi}_{X,Y \otimes Z} = (\id_Y \otimes \bar{\psi}_{XZ})(\bar{\psi}_{XY} \otimes \id_Z);
  \end{equation}
  the other one is similar.
  Consider the following diagram:
  \begin{equation*}
    \begin{tikzpicture}
      [baseline=(current bounding box.center)
      description/.style={fill=white,inner sep=2pt}]
      \matrix (m) [matrix of math nodes, row sep=3em,
      column sep=3em, text height=2ex, text depth=0.25ex]
      {
        X \otimes Y \otimes Z & \bar{\bar{X}} \otimes \bar{\bar{Y \otimes Z}} & \bar{\bar{Y \otimes Z}} \otimes \bar{\bar{X}} & Y \otimes Z \otimes X \\
        & \bar{\bar{X}} \otimes \bar{\bar{Z} \otimes \bar{Y}} & \bar{\bar{Z} \otimes \bar{Y}} \otimes \bar{\bar{X}} & \\
        \bar{\bar{X}} \otimes \bar{\bar{Y}} \otimes Z & \bar{\bar{X}} \otimes \bar{\bar{Y}} \otimes \bar{\bar{Z}} & \bar{\bar{Y}} \otimes \bar{\bar{Z}} \otimes \bar{\bar{X}} & \\
        \bar{\bar{Y}} \otimes \bar{\bar{X}} \otimes Z & \bar{\bar{Y}} \otimes \bar{\bar{X}} \otimes \bar{\bar{Z}} & & \\
        Y \otimes X \otimes Z & Y \otimes \bar{\bar{X}} \otimes \bar{\bar{Z}} & Y \otimes \bar{\bar{Z}} \otimes \bar{\bar{X}} & \\
      };
      \path[->,font=\scriptsize]
      (m-1-1) edge node[auto] {$\sigma^{-1} \otimes \sigma^{-1}$} (m-1-2)
      (m-1-1) edge node[auto,swap] {$\sigma^{-1} \otimes \sigma^{-1} \otimes \id$} (m-3-1)
      (m-1-2) edge node[auto] {$\xi$} (m-1-3)
      (m-1-2) edge node[auto,swap] {$\id \otimes \bar{\rho}$} (m-2-2)
      (m-1-3) edge node[auto] {$\sigma \otimes \sigma$} (m-1-4)
      (m-1-3) edge node[auto,swap] {${\bar{\rho} \otimes \id}$} (m-2-3)
      (m-2-2) edge node[auto] {$\xi$} (m-2-3)
      (m-2-2) edge node[auto,swap] {$\id \otimes \rho$} (m-3-2)
      (m-2-3) edge node[auto,swap] {$\rho \otimes \id$} (m-3-3)
      (m-3-1) edge node[auto] {$\id \otimes \id \otimes \sigma^{-1}$} (m-3-2)
      (m-3-1) edge node[auto,swap] {$\xi \otimes \id$} (m-4-1)
      (m-3-2) edge node[auto,swap] {$\xi \otimes \id$} (m-4-2)
      (m-3-3) edge node[auto,swap] {$\sigma \otimes \sigma \otimes \sigma$} (m-1-4)
      (m-3-3) edge node[auto] {$\sigma \otimes \id \otimes \id$} (m-5-3)
      (m-4-1) edge node[auto] {$\id \otimes \id \otimes \sigma^{-1}$} (m-4-2)
      (m-4-1) edge node[auto,swap] {$\sigma \otimes \sigma \otimes \id$} (m-5-1)
      (m-4-2) edge node[auto] {$\id \otimes \xi$} (m-3-3)
      (m-4-2) edge node[auto] {$\sigma \otimes \xi$} (m-5-3)
      (m-4-2) edge node[auto,swap] {$\sigma \otimes \id \otimes \id$} (m-5-2)
      (m-5-1) edge node[auto,swap] {$\id \otimes \sigma^{-1} \otimes \sigma^{-1}$} (m-5-2)

      (m-5-2) edge node[auto,swap] {$\id \otimes \xi$} (m-5-3)
      (m-5-3) edge [bend right=30] node[auto,swap] {$\id \otimes \sigma \otimes \sigma$} (m-1-4);
    \end{tikzpicture}    
  \end{equation*}
  Although this looks rather unpleasant, it is mostly harmless.
  Almost all of the pieces commute tautologically.
  The only pieces that are not obvious are the square in the top middle part of the diagram, and the pentagon immediately below that square.
  The square commutes by Lemma \ref{lem:xis-are-natural}, and the pentagon commutes by Lemma \ref{lem:xi-hexagons}.

  To complete the proof, note that the left-hand side of \eqref{eq:conjugated-braiding-hexagon} is exactly the composition along the top row of the diagram, while the right-hand side of \eqref{eq:conjugated-braiding-hexagon} is the composition down the left side, across the bottom, and up to the top right corner.
\end{proof}

\subsection{Quasitriangular Hopf $*$-algebras}
\label{sec:quasitriangularity}

Quasitriangular Hopf algebras are important because their module categories are automatically equipped with a braiding.
Our goal in \S \ref{sec:quasitriangularity} is to understand what the conjugate braiding $\bar{\psi}$ looks like when $H$ is a quasitriangular Hopf $*$-algebra.
We begin by recalling the definition and properties of quasitriangular Hopf algebras.

Recall that a Hopf algebra $H$ is \emph{quasitriangular} \cite{KliSch97}*{Chapter 8} if there is an invertible element $\rmat \in H \otimes H$ such that
\begin{gather}
  \label{eq:rmatrix-coprod-reversal}
  \rmat \cop(a) \rmat^{-1} = \cop^{op}(a) \overset{\mathrm{def}}{=} \tau\circ \cop (a),\\
  \intertext{and}
  \label{eq:rmatrix-qybe-coproduct-formula}
  (\cop \otimes \id)(\rmat) = \rmat_{13} \rmat_{23}, \quad (\id \otimes \cop)(\rmat) = \rmat_{13} \rmat_{12},
\end{gather}
where $\tau$ is the tensor flip.
The element $\rmat$ is called a \emph{universal R-matrix} for $H$.
In \eqref{eq:rmatrix-qybe-coproduct-formula} we are using the so-called \emph{leg-numbering notation}: if $\rmat = \sum_j x_j \otimes y_j$, then $\rmat_{12} = \sum_j x_j \otimes y_j \otimes 1$, $\rmat_{13} = \sum_j x_j \otimes 1 \otimes y_j$, and $\rmat_{23} = \sum_j 1 \otimes  x_j \otimes y_j$.
We also denote $\rmat_{21} = \tau(\rmat) = \sum_j y_j \otimes x_j$.

The important point here is that if $H$ is quasitriangular, then $\hmod$ acquires a braiding from the action of the R-matrix.
Before describing the braiding, we record here for later use some of the consequences of quasitriangularity.
The proofs of these statements can be found in \S 8.1 of \cite{KliSch97}.

\begin{prop}
  \label{prop:quasitriangular-properties}
  Suppose that $H$ is a quasitriangular Hopf algebra with universal R-matrix $\rmat$.
  Then the following hold:
  \begin{gather}
    \label{eq:rmatrix-qybe}
    \rmat_{12} \rmat_{13} \rmat_{23} = \rmat_{23} \rmat_{13} \rmat_{12} \\
    \label{eq:rmatrix-counit-equation}
    (\counit \otimes \id)(\rmat) = (\id \otimes \counit)(\rmat) = 1 \\
    \label{eq:rmatrix-antipode}
    (S \otimes \id)(\rmat) = \rmat^{-1}, \quad (\id \otimes S)(\rmat^{-1}) = \rmat, \quad (S \otimes S)(\rmat) = \rmat.
  \end{gather}
  Furthermore, the element $u = m(S \otimes \id)(\rmat_{21})$ is invertible in $H$ with inverse $u^{-1} = m(\id \otimes S^2)(\rmat_{21})$ (here $m$ denotes the multiplication map of $H$), and the element $uS(u) = S(u)u$ is central in $H$.
  The antipode of $H$ is invertible, and we have
  \begin{equation}
    \label{eq:rmatrix-antipode-squared}
    S^2(a) = uau^{-1}, \quad S^{-1}(a) = u^{-1}S(a)u
  \end{equation}
  for all $a \in H$.
\end{prop}
The relation \eqref{eq:rmatrix-qybe} is called the \emph{quantum Yang-Baxter equation}.

If $U,V$ are any two $H$-modules, then $\rmat$ acts in $U \otimes V$ coordinate-wise.
Then we have the following result \cite{KliSch97}*{\S 8.1.2}:

\begin{prop}
  \label{prop:rmatrix-braiding}
  Suppose that $H$ is a quasitriangular Hopf algebra with universal R-matrix $\rmat$.
  For any $U,V \in \hmod$, define 
  \[ \psi_{UV} = \tau \circ \rmat : U \otimes V \to V \otimes U.  \]
  Then $\psi = (\psi_{UV})_{U,V \in \hmod}$ is a braiding on $\hmod$.
\end{prop}

Now that the prerequisites are in place, we can begin to explore the effect of conjugation on the braiding coming from the R-matrix.

\begin{notn}
  \label{notn:star-in-tensor-prod}
  We extend the $*$-structure of $H$ to a $*$-structure on $H \otimes H$ coordinate-wise.
  For a simple tensor $a \otimes b \in H \otimes H$, we have $(a \otimes b)^* = a^* \otimes b^*$.
  In particular, for $\rmat = \sum_j x_j \otimes y_j$ we have 
  \[  \rmat^* = \sum_j x_j^* \otimes y_j^*. \]
\end{notn}

The following definition is inspired by \cite{KliSch97}{Definition 2 in \S 10.1.1}:
\begin{dfn}
  \label{dfn:rmatrix-real-inverse-real}
  Suppose that $H$ is a quasitriangular Hopf $*$-algebra with universal R-matrix $\rmat$.
  We say that $\rmat$ is \emph{real} if $\rmat^* = \rmat_{21}$.
  We say that $\rmat$ is \emph{inverse real} if $\rmat^* = \rmat^{-1}$.
\end{dfn}

When the R-matrix is either real or inverse real, we can draw some conclusions about the conjugated braiding:

\begin{prop}
  \label{prop:rmatrix-reality-and-conjugated-braiding}
  Suppose that $H$ is a quasitriangular Hopf $*$-algebra with universal R-matrix $\rmat$.
  Let $\psi$ be the braiding on $\hmod$ as in Proposition \ref{prop:rmatrix-braiding}
  \begin{enumerate}[(a)]
  \item If $\rmat$ is real, then $\bar{\psi} = \psi$, i.e.~for any $U,V \in \hmod$ we have $\bar{\psi}_{UV} = \psi_{UV}$.
  \item If $\rmat$ is inverse real, then $\bar{\psi} = \psi^{-1}$, i.e.~for any $U,V \in \hmod$ we have $\bar{\psi}_{UV} = \psi_{VU}^{-1}$.
  \end{enumerate}
\end{prop}

\begin{proof}
  Let us write $\rmat = x_j \otimes y_j$ with implied summation.
  First we compute the maps $\xi_{VW}$ that we defined in \eqref{eq:conj-braiding-step1}.
  For $v \in V$, $w \in W$ we have
  \begin{align*}
    \xi_{VW} (\bar{v} \otimes \bar{w}) & = \rho_{VW} \left( \bar{\psi_{WV}(w \otimes v)} \right) \\
    & = \rho_{VW} \left(\bar{\tau \circ \rmat(w \otimes v)}\right) \\
    & = \rho_{VW} \left( \bar{(y_j \rhd v) \otimes (x_j \rhd w)}\right) \\
    & = \bar{(x_j \rhd w)} \otimes \bar{(y_j \rhd v)} \\
    & = (S(x_j)^* \rhd \bar{w}) \otimes (S(y_j)^* \rhd \bar{v}) \\
    & = \left[ (S \otimes S)(\rmat)^* \right](\bar{w} \otimes \bar{v}) \\
    & = \rmat^* (\bar{w} \otimes \bar{v}) \\
    & = (\tau \circ \rmat_{21}^*)(\bar{v} \otimes \bar{w}).
  \end{align*}
  In other words, we have shown that $\xi_{VW} = \tau \circ \rmat_{21}^*$.

  When $\rmat$ is real, then we have 
  \[ \xi_{VW} = \tau \circ \rmat_{21}^* = \tau \circ \rmat  = \psi_{\bar{V} \, \bar{W}}. \]
  Applying naturality of the braiding $\psi$ to the morphisms $\sigma_{V}$ and $\sigma_W$ (and suppressing subscripts for readability), we have
  \[  \bar{\psi}_{VW} = (\sigma \otimes \sigma) \xi_{\bar{V} \, \bar{W}} (\sigma^{-1}\otimes \sigma^{-1}) = (\sigma \otimes \sigma) \psi_{\bar{\bar{V}} \, \bar{\bar{W}}} (\sigma^{-1}\otimes \sigma^{-1}) = \psi_{VW}.\]
  This establishes (a).

  When $\rmat$ is inverse real, then our initial computation gives
  \[  \xi_{VW} = \tau \circ \rmat_{21}^* = \tau \circ \rmat_{21}^{-1} = \rmat^{-1} \circ \tau = \psi_{\bar{W} \, \bar{V}}^{-1};    \]
  as above, using naturality of the braiding $\psi$, we conclude that $\bar{\psi}_{VW} = \psi_{WV}^{-1}$.
  This establishes (b).
\end{proof}

\subsection{Braidings and $*$-structures}
\label{sec:braidings-and-star-structures}

For this subsection we will assume that $H$ is a quasitriangular Hopf $*$-algebra, so that $\hmod$ is braided.
As above, we denote the braiding coming from the R-matrix by $\psi$.
Recall from Proposition \ref{prop:starmods-tensor-powers} that if $V$ is a $*$-module, then $V \otimes V$ is also a $*$-module, with $*$-structure given by $(\bar{u \otimes v})^* = \bar{v}^* \otimes \bar{u}^*$.
It therefore makes sense to ask: under what circumstances is the braiding map $\psi_{VV}$ a morphism of $*$-modules?
In other words, when does the following diagram commute:
\begin{equation}
  \label{eq:braiding-is-a-star-map}
  \begin{CD}
    \bar{V \otimes V} @>*>> V \otimes V \\
    @V{\bar{\psi_{VV}}}VV @VV{\psi_{VV}}V \\
    \bar{V \otimes V} @>>{*}> V \otimes V
  \end{CD}
\end{equation}
Note that on the left-hand side of \eqref{eq:braiding-is-a-star-map} we really mean the complex conjugate of the map $\psi_{VV}$, not the conjugated braiding $\bar{\psi}_{VV}$ (which wouldn't make sense anyway).

\begin{prop}
  \label{prop:braiding-is-star-map-real-case}
  Let $V \in \hmod$ be a $*$-module, and give $V \otimes V$ the $*$-structure described in Proposition \ref{prop:starmods-tensor-powers}. 
  If the R-matrix $\rmat$ of $H$ is real, then the diagram \eqref{eq:braiding-is-a-star-map} commutes, so $\psi_{VV}$ is a morphism of $*$-modules.
\end{prop}

\begin{proof}
  Again, let us denote the universal R-matrix by $\rmat = x_j \otimes y_j$, with implied summation.
  Let $u,v \in V$.
  On the one hand, we have
  \begin{align*}
    \psi_{VV} \left( (\bar{u \otimes v})^*  \right) & = \psi_{VV}(\bar{v}^* \otimes \bar{u}^*) \\
    & = (\tau \circ \rmat)(\bar{v}^* \otimes \bar{u}^*) \\
    & = (y_j \rhd \bar{u}^*) \otimes (x_j \rhd \bar{v}^*) \\
    & = \rmat_{21} (\bar{u}^* \otimes \bar{v}^*).
  \end{align*}
  On the other hand, we have
  \begin{align*}
    \left(  \bar{\psi_{VV}}(\bar{u \otimes v})  \right)^* & = \left(  \bar{\psi_{VV}(u \otimes v)}  \right)^* \\
    & = \left(  \bar{(\tau \circ \rmat)  (u \otimes v)}  \right)^* \\
    & = \left( \bar{(y_j \rhd v) \otimes (x_j \rhd u)}   \right)^* \\
    & = \bar{(x_j \rhd u)}^* \otimes \bar{(y_j \rhd v)}^* \\
    & = (S(x_j)^* \rhd \bar{u}^*) \otimes (S(y_j)^* \rhd \bar{v}^*) \\
    & = (S \otimes S)(\rmat)^*(\bar{u}^* \otimes \bar{v}^*) \\
    & = \rmat^*(\bar{u}^* \otimes \bar{v}^*).
  \end{align*}
  If $\rmat$ is real, then $\rmat^* = \rmat_{21}$ by definition, so the two expressions agree, and thus \eqref{eq:braiding-is-a-star-map} commutes.
  Hence $\psi_{VV}$ is a $*$-map.
\end{proof}

\end{document}